\documentclass[10pt]{amsart}
\usepackage{amssymb}
\usepackage{bm}
\usepackage{graphicx}
\usepackage[centertags]{amsmath}
\usepackage{amsfonts}
\usepackage{amsthm}
\usepackage{graphicx}
\usepackage[toc,page]{appendix}
\newtheorem{thm}{Theorem}[section]
\newtheorem{cor}[thm]{Corollary}
\newtheorem{lem}[thm]{Lemma}

\newtheorem{prop}[thm]{Proposition}
\newtheorem{claim}[thm]{Claim}

\newtheorem{fact}[thm]{Fact}
\newtheorem{defn}[thm]{Definition}

\theoremstyle{definition}

\newcommand{\nn}{\mathbb{N}}
\newcommand{\ee}{\varepsilon}

\newcommand{\suc}{\mathrm{Succ}}
\newcommand{\immsuc}{\mathrm{ImmSucc}}
\newcommand{\strong}{\mathrm{Str}}
\newcommand{\ci}{\mathrm{I}}
\newcommand{\mil}{\mathrm{Mil}}

\newcommand{\udhl}{\mathrm{UDHL}}
\newcommand{\dens}{\mathrm{dens}}
\newcommand{\ave}{\mathbb{E}}

\newcommand{\bfr}{\mathbf{r}}
\newcommand{\bfs}{\mathbf{s}}
\newcommand{\bft}{\mathbf{t}}
\newcommand{\bfu}{\mathbf{u}}

\newcommand{\bfz}{\mathbf{z}}

\newcommand{\bfcb}{\mathbf{B}}

\newcommand{\bfcf}{\mathbf{F}}
\newcommand{\bfcg}{\mathbf{G}}

\newcommand{\bfci}{\mathbf{I}}
\newcommand{\bfcp}{\mathbf{P}}

\newcommand{\bfcr}{\mathbf{R}}
\newcommand{\bfcs}{\mathbf{S}}
\newcommand{\bfct}{\mathbf{T}}
\newcommand{\bfcu}{\mathbf{U}}

\newcommand{\bfcz}{\mathbf{Z}}

\newcommand{\meg}{\geqslant}
\newcommand{\mik}{\leqslant}
\newcommand{\lex}{<_{\mathrm{lex}}}
\newcommand{\con}{\smallfrown}

\begin{document}

\title{Measurable events indexed by products of trees}

\author{Pandelis Dodos, Vassilis Kanellopoulos and Konstantinos Tyros}

\address{Department of Mathematics, University of Athens, Panepistimiopolis 157 84, Athens, Greece}
\email{pdodos@math.uoa.gr}

\address{National Technical University of Athens, Faculty of Applied Sciences,
Department of Mathematics, Zografou Campus, 157 80, Athens, Greece}
\email{bkanel@math.ntua.gr}

\address{Department of Mathematics, University of Toronto, Toronto, Canada M5S 2E4}
\email{k.tyros@utoronto.ca}

\thanks{The first named author was supported by NSF grant DMS-0903558.}

\maketitle


\begin{abstract}
A tree $T$ is said to be homogeneous if it is uniquely rooted and there exists an integer $b\meg 2$, called the branching
number of $T$, such that every $t\in T$ has exactly $b$ immediate successors. A vector homogeneous tree $\bfct$ is a finite
sequence $(T_1,...,T_d)$ of homogeneous trees and its level product $\otimes\bfct$ is the subset of the Cartesian product
$T_1\times ...\times T_d$ consisting of all finite sequences $(t_1,...,t_d)$ of nodes having common length.

We study the behavior of measurable events in probability spaces indexed by the level product $\otimes\bfct$ of a vector
homogeneous tree $\bfct$. We show that, by refining the index set to the level product $\otimes\bfcs$ of a vector strong
subtree $\bfcs$ of $\bfct$, such families of events become highly correlated. An analogue of Lebesgue's density Theorem
is also established which can be considered as the ``probabilistic" version of the density Halpern--L\"{a}uchli Theorem.
\end{abstract}


\section{Introduction}

\numberwithin{equation}{section}

\subsection{Overview}

The present paper is devoted to the analysis of a phenomenon encountered in Ramsey Theory and concerns the structure of
measurable events in probability spaces indexed by a \textit{Ramsey space} \cite{C,To}. The phenomenon is most transparently
seen when the events are indexed by the natural numbers $\nn$, an archetypical Ramsey space.  Specifically, let $(\Omega,\Sigma,\mu)$
be a probability space and assume that we are given a family $\{A_i:i\in\nn\}$ of measurable events in $(\Omega,\Sigma,\mu)$
satisfying $\mu(A_i)\meg \ee>0$ for every $i\in\nn$. Using the classical Ramsey Theorem \cite{Ra} and elementary probabilistic
estimates, it is easy to see that for every $0<\theta<\ee$ there exists an infinite subset $L$ of $\nn$ such that for every
integer $n\meg 1$ and every subset $F$ of $L$ of cardinality $n$ we have
\begin{equation} \label{ei1}
\mu \Big(\bigcap_{i\in F} A_i\Big) \meg \theta^n.
\end{equation}
In other words, the events in the family $\{A_i:i\in L\}$ are at least as correlated as if they were independent.

A natural problem, which is of combinatorial and analytical importance, is to decide whether the aforementioned result is valid
if the events are indexed by another Ramsey space $\mathbb{S}$. Namely, given a family $\{A_s:s\in\mathbb{S}\}$ of measurable
events in a probability space $(\Omega,\Sigma,\mu)$ satisfying $\mu(A_s)\meg\ee>0$ for every $s\in\mathbb{S}$, is it possible
to find a ``substructure" $\mathbb{S}'$ of $\mathbb{S}$ such that the events in the family $\{A_s:s\in\mathbb{S}'\}$ are highly
correlated? And if yes, then can we get explicit (and, hopefully, optimal) lower bounds for their joint probability? Of course,
the notion of ``substructure" will depend on the nature of the given index set $\mathbb{S}$.

The significance of this problem will be mostly appreciated when one considers the Ramsey space $W(\mathbb{A})$ of all finite
words over a nonempty finite alphabet $\mathbb{A}$. Specifically, it was shown by H. Furstenberg and Y. Katznelson in \cite{FK}
that for every integer $k\meg 2$ and every $0<\ee\mik 1$ there exists a strictly positive constant $\theta(k,\ee)$ with the
following property. If $\mathbb{A}$ is an alphabet with $k$ letters and $\{A_w:w\in W(\mathbb{A})\}$ is a family of measurable
events in a probability space $(\Omega,\Sigma,\mu)$ satisfying $\mu(A_w)\meg\ee$ for every $w\in W(\mathbb{A})$, then there
exists a combinatorial line $\mathbb{L}$ (see \cite{HJ}) such that
\begin{equation} \label{ei2}
\mu\Big( \bigcap_{w\in\mathbb{L}} A_w\Big) \meg \theta(k,\ee).
\end{equation}
This statement is easily seen to be equivalent to the density Hales--Jewett Theorem, a fundamental result of Ramsey Theory. 
Although powerful, the arguments in \cite{FK} are qualitative in nature and give no estimate on the constant $\theta(k,\ee)$. 
Explicit lower bounds can be extracted, however, from \cite{DKT4}.

\subsection{The main results}

In \cite{DKT1} we studied the above problem when the events are indexed by a homogeneous tree; we recall that a tree $T$ is said
to be \textit{homogeneous} if it is uniquely rooted and there exists an integer $b\meg 2$, called the \textit{branching number}
of $T$, such that every $t\in T$ has exactly $b$ immediate successors. Our goal in this paper is to extend this analysis to
the higher-dimensional setting, namely when we deal with events indexed by the level product of a vector homogeneous tree. 
We recall that a \textit{vector homogeneous tree} $\bfct$ is a finite sequence $(T_1,...,T_d)$ of homogeneous trees and its
\textit{level product} $\otimes\bfct$ is the subset of the Cartesian product $T_1\times ...\times T_d$ consisting of all
finite sequences $(t_1,...,t_d)$ of nodes having common length. In particular, $\otimes\bfct(n)$ stands for the standard
Cartesian product $T_1(n)\times ...\times T_d(n)$.

In the context of trees the most natural notion of ``substructure" is that of a \textit{strong subtree}. We recall that a
subtree $S$ of a uniquely rooted tree $T$ is said to be strong provided that: (a) $S$ is uniquely rooted and balanced
(that is, all maximal chains of $S$ have the same cardinality), (b) every level of $S$ is a subset of some level of $T$,
and (c) for every non-maximal node $s\in S$ and every immediate successor $t$ of $s$ in $T$ there exists a unique immediate
successor $s'$ of $s$ in $S$ with $t\mik s'$.  The \textit{level set} of a strong subtree $S$ of a tree $T$ is the set of
levels of $T$ containing a node of $S$. The concept of a strong subtree is, of course, extended to vector trees. Specifically,
a \textit{vector strong subtree} of a vector tree $\bfct=(T_1,...,T_d)$ is just a finite sequence $\bfcs=(S_1,...,S_d)$ of
strong subtrees of $(T_1,...,T_d)$ having common level set.

\subsubsection{The continuous case}

We are ready to state the first main result of the paper.
\begin{thm} \label{it1}
For every integer $d\meg 1$, every $b_1,...,b_d\in\nn$ with $b_i\meg 2$ for all $i\in\{1,...,d\}$, every integer $n\meg 1$
and every $0<\ee\mik 1$ there exists a strictly positive constant $c(b_1,...,b_d|n,\ee)$ with the following property. 
If $\bfct=(T_1,...,T_d)$ is a vector homogeneous tree such that the branching number of $T_i$ is $b_i$ for all $i\in\{1,...,d\}$
and $\{A_\bft:\bft\in \otimes\bfct\}$ is a family of measurable events in a probability space $(\Omega,\Sigma,\mu)$ satisfying
$\mu(A_\bft)\meg\ee$ for every $\bft\in \otimes\bfct$, then there exists a vector strong subtree $\bfcs=(S_1,...,S_d)$ of
$\bfct$ of infinite height such that for every integer $n\meg 1$ and every subset $F$ of the level product $\otimes\bfcs$
of $\bfcs$ of cardinality $n$ we have
\begin{equation} \label{ei3}
\mu\Big( \bigcap_{\bft\in F} A_\bft\Big) \meg c(b_1,...,b_d|n,\ee).
\end{equation}
\end{thm}
Theorem \ref{it1} is the higher-dimensional extension of \cite[Theorem 1]{DKT1} where the case of a single homogeneous tree
was treated. In particular, in \cite{DKT1} it was shown, among others, that
\begin{equation} \label{e1new}
c(b|n,\ee) \meg \ee^{2^{2bn}}.
\end{equation}
We point out that the proof of Theorem \ref{it1} is also effective and yields explicit lower bounds for the constants
$c(b_1,...,b_d|n,\ee)$. These estimates, however, are admittedly rather weak and it is an important problem to obtain ``civilized"
bounds for the relevant constants. A crucial ingredient of the argument is \cite[Theorem 3]{DKT2}. It is used in the proof of the
following proposition which is, possibly, of independent interest.
\begin{prop} \label{ip2}
For every integer $d\meg 1$, every $b_1,...,b_d\in\nn$ with $b_i\meg 2$ for all $i\in\{1,...,d\}$ and every  $0<\ee\mik 1$
there exist an integer $\mathrm{Cor}(b_1,...,b_d|\ee)$ and a strictly positive constant $\xi(b_1,...,b_d|\ee)$ with the 
following property. If $\bfct=(T_1,...,T_d)$ is a finite vector homogeneous tree of height at least $\mathrm{Cor}(b_1,...,b_d|\ee)$
such that the branching number of $T_i$ is $b_i$ for all $i\in\{1,...,d\}$ and $\{A_\bft:\bft\in \otimes\bfct\}$ is a family
of measurable events in a probability space $(\Omega,\Sigma,\mu)$ satisfying  $\mu(A_\bft)\meg\ee$ for every $\bft\in \otimes\bfct$,
then there exists a vector strong subtree $\bfcf$ of $\bfct$ of height $2$ such that
\begin{equation} \label{ei4}
\mu\Big( \bigcap_{\bft\in \otimes\bfcf} A_\bft\Big) \meg \xi(b_1,...,b_d|\ee).
\end{equation}
\end{prop}

\subsubsection{The discrete case}

To proceed with our discussion we need, first, to introduce some definitions. To motivate the reader, let us assume that
we are given a family $\{A_{\bft}:\bft\in\otimes\bfct\}$ of Lebesgue measurable subsets of the unit interval indexed by the
level product of a finite vector homogeneous tree $\bfct$. Using a standard approximation argument and up to negligible errors,
for every $n< h(\bfct)$ it is possible to find an integer $l_n$ such that every event in the family $\{A_{\bft}:\bft\in\otimes\bfct(n)\}$ 
belongs to the algebra generated by all dyadic intervals of length $2^{-l_n}$. This observation leads to the following definition.
\begin{defn} \label{id3}
Let $\bfct$ be a finite vector homogeneous tree and $W$ a homogeneous tree. We say that a map $D:\otimes\bfct\to \mathcal{P}(W)$
is a \emph{level selection} if there exists a subset $L(D)=\{\ell_0< ... < \ell_{h(\bfct)-1}\}$ of $\nn$, called the \emph{level set}
of $D$, such that for every $n<h(\bfct)$ and every $\bft\in\otimes \bfct(n)$ we have that $D(\bft)\subseteq W(l_n)$.

For every level selection $D:\otimes\bfct\to \mathcal{P}(W)$ the \emph{height} $h(D)$ of $D$ is defined to be the height $h(\bfct)$
of the finite vector homogeneous tree $\bfct$. The \emph{density} $\delta(D)$ of $D$ is the quantity defined by
\begin{equation} \label{ei5}
\delta(D)= \min\big\{\dens\big(D(\bft)\big):\bft\in\otimes\bfct\big\}.
\end{equation}
\end{defn}
We remark that if $W$ is a tree and $F\subseteq W(\ell)$ for some $\ell\in\nn$, then the \textit{density} of $F$ is defined by
\begin{equation} \label{ei6}
\dens(F)=\frac{|F|}{|W(\ell)|}.
\end{equation}
More generally, if $m\in\nn$ with $m\mik \ell$ and $w\in W(m)$, then the \textit{density of $F$ relative to $w$} is defined by
\begin{equation} \label{ei7}
\dens(F \ | \ w)=\frac{|F\cap \suc_W(w)|}{|W(\ell)\cap \suc_W(w)|}
\end{equation}
where $\suc_W(w)$ stands for the set of all successors of $w$ in $W$. Notice that the density of the set $F$ relative to the node
$w$ is the usual density of $F$ when restricted to the subtree $\suc_W(w)$.

It follows from the above discussion that a level selection $D:\otimes\bfct\to \mathcal{P}(2^{<\nn})$ of density $\ee$, where
$2^{<\nn}$ stands for the dyadic tree, is just the discrete version of a family $\{A_\bft:\bft\in\otimes\bfct\}$ of Lebesgue
measurable subsets of the unit interval each having measure at least $\ee$. We should point out that, beside their intrinsic
interest, level selections arose quite naturally in the proof of the density Halpern--L\"{a}uchli Theorem \cite{DKK,DKT2}. 
In fact, the density Halpern--L\"{a}uchli Theorem is essentially a statement concerning the structure of level selections.

By Proposition \ref{ip2}, for every finite vector homogeneous tree $\bfct=(T_1,...,T_d)$ and every level selection
$D:\otimes\bfct\to \mathcal{P}(W)$ of density $\ee$ and of sufficiently large height, it is possible to find a vector strong
subtree $\bfcf$ of $\bfct$ of height $2$ such that the density of the set
\begin{equation} \label{ei8}
D_{\bfcf}:=\bigcap_{\bft\in\otimes\bfcf(1)} D(\bft)
\end{equation}
is at least $c'$, where $c'$ is an absolute constant depending only on the branching numbers of the trees $T_1,...,T_d$ and
the given $\ee$. This fact is certainly useful but it gives us no information on how the set $D_{\bfcf}$ is distributed inside
the tree $W$. To clarify what we mean exactly about the distribution of the set $D_{\bfcf}$ it is convenient to introduce the
following definition.
\begin{defn} \label{id4}
Let $D:\otimes\bfct\to \mathcal{P}(W)$ be a level selection. Also let $\bfcf$ be a vector strong subtree of $\bfct$ of height $2$,
$w\in W$ and $0<\theta\mik 1$. We say that the pair $(\bfcf,w)$ is \emph{strongly $\theta$-correlated with respect to $D$} if the
following conditions are satisfied.
\begin{enumerate}
\item[(1)] We have that $w\in D(\bfr)$ where $\bfr$ is the root of $\bfcf$.
\item[(2)] For every immediate successor $s$ of $w$ in $W$ the density of the set $D_{\bfcf}$ relative to $s$ is at least $\theta$.
\end{enumerate}
\end{defn}
Roughly speaking, if a pair $(\bfcf,w)$ is strongly $\theta$-correlated with respect to $D$, then the set $D_{\bfcf}$ looks like
a randomly chosen subset of the subtree $\suc_W(w)$ of density $\theta$. Such a node $w$ is expected to exist, by Lebesgue's 
density Theorem. The main point guaranteed by Definition \ref{id4} is that the desired node $w$ will be found in an \textit{a priori}
given set; the set $D(\bfr)$ where $\bfr$ is the root of $\bfcf$.

We are now ready to state the second main result of the paper.
\begin{thm} \label{it5}
For every integer $d\meg 1$, every $b_1,...,b_d,b_{d+1}\in\nn$ with $b_i\meg 2$ for all $i\in\{1,...,d+1\}$ and every $0<\ee\mik 1$
there exist an integer $\mathrm{StrCor}(b_1,...,b_d,b_{d+1}|\ee)$ and a strictly positive constant $\theta(b_1,...,b_d,b_{d+1}|\ee)$
with the following property. If $\bfct=(T_1,...,T_d)$ is a finite vector homogeneous tree such that the branching number of $T_i$
is $b_i$ for all $i\in\{1,...,d\}$, $W$ is a homogeneous tree with branching number $b_{d+1}$ and $D:\otimes\bfct\to\mathcal{P}(W)$
is a level selection of density $\ee$ and of height at least $\mathrm{StrCor}(b_1,...,b_d,b_{d+1}|\ee)$, then there exist a vector
strong subtree $\bfcf$ of $\bfct$ of height $2$ and a node $w\in W$ such that the pair $(\bfcf,w)$ is strongly
$\theta(b_1,...,b_d,b_{d+1}|\ee)$-correlated with respect to $D$.
\end{thm}
Theorem \ref{it5} is the most demanding result of the paper. Its proof follows a density increment strategy -- a powerful method
pioneered by K. F. Roth \cite{Roth} -- and is based, in an essentially way, on \cite[Theorem 3]{DKT2}. The argument is effective.
In particular, we provide explicit estimates for all numerical invariants appearing in Theorem \ref{it5}.

\subsection{Consequences} 

We proceed to discuss the relation between Theorem \ref{it5} and the infinite version of the density Halpern--L\"{a}uchli Theorem.
Let us recall, first, the statement of this result (see \cite[Theorem 2]{DKK}).
\begin{thm} \label{dhl}
For every integer $d\meg 1$ we have that $\mathrm{DHL}(d)$ holds, i.e. for every vector homogeneous tree $\bfct=(T_1,...,T_d)$ and
every subset $D$ of the level product $\otimes\bfct$ of $\bfct$ satisfying
\begin{equation} \label{e1-final section}
\limsup_{n\to\infty} \frac{|D\cap \otimes\bfct(n)|}{|\!\otimes\bfct(n)|}>0
\end{equation}
there exists a vector strong subtree $\bfcs$ of $\bfct$ of infinite height whose level product is contained in $D$.
\end{thm}
We point out that the case ``$d=1$" of Theorem \ref{dhl} is due to R. Bicker and B. Voigt \cite{BV} and we refer the reader to
\cite[\S 1]{DKK} for a discussion on the history of this result.

Also we need to extend Definition \ref{id3} to the infinite-dimensional setting. Specifically, let $0<\ee\mik 1$, $\bfct$ a vector
homogeneous tree and $W$ a homogeneous tree. We say that a map $D:\otimes\bfct\to \mathcal{P}(W)$ is an \textit{$\ee$-dense level selection}
if there exists an infinite subset $L(D)=\{\ell_0<\ell_1<...\}$ of $\nn$ such that for every $n\in\nn$ and every $\bft\in\otimes\bfct(n)$
we have $D(\bft)\subseteq W(\ell_n)$ and $\dens\big(D(\bft)\big)\meg\ee$. This notion was introduced in \cite{DKK} and it was
critical for the proof of Theorem \ref{dhl}. 

We have the following theorem.
\begin{thm} \label{cor1-final section}
Let $d\in\nn$ with $d\meg 1$ and $b_1,...,b_d,b_{d+1}\in\nn$ with $b_i\meg 2$ for all $i\in \{1,...,d+1\}$. Also let $0<\ee\mik 1$,
$\bfct=(T_1,...,T_d)$ a vector homogeneous tree such that the branching number of $T_i$ is $b_i$ for all $i\in \{1,...,d\}$,
$W$ a homogeneous tree with branching number $b_{d+1}$ and $D:\otimes\bfct\to \mathcal{P}(W)$ an $\ee$-dense level selection. 
Then there exist a vector strong subtree $\bfcs$ of $\bfct$ of infinite height and for every $\bfs\in\otimes\bfcs$ a node 
$w_{\bfs}\in D(\bfs)$ such that for every strong subtree $\bfcf$ of $\bfcs$ of height $2$ with $\bfcf(0)=\bfs$ the pair
$(\bfcf,w_{\bfs})$ is strongly $\theta(b_1,...,b_d,b_{d+1}|\ee)$-correlated with respect to $D$, where the constant
$\theta(b_1,...,b_d,b_{d+1}|\ee)$ is as in Theorem \ref{it5}.
\end{thm}
Theorem \ref{cor1-final section} follows from Theorem \ref{it5} and Milliken's Theorem \cite{Mi2} using fairly standard
arguments; we leave the details to the interested reader. It is a quantitative strengthening of \cite[Corollary 10]{DKK}
where a similar result was obtained but no estimate was given for the relevant constants. While we find such an improvement
interesting per se, our interest in Theorem \ref{cor1-final section} was, mainly, for utilitarian reasons. Specifically, 
Theorem  \ref{cor1-final section} can be used to derive Theorem \ref{dhl} by plugging it, as pigeonhole principle, in the
recursive construction presented in \cite[\S 5]{DKK}.

Thus we see that Theorem \ref{it5}, which is an entirely \textit{finitary} statement but of \textit{probabilistic} nature, 
can be used to derive the corresponding infinite-dimensional result. This methodology has been applied successfully to related
problems in Ramsey Theory -- see, in particular, \cite{DKT3}.

\subsection{Organization of the paper}

The paper is organized as follows. In \S 2 we set up our notation and terminology, and we recall some tools needed for
the proofs of the main results. In \S 3 we give the proofs of Theorem \ref{it1} and Proposition \ref{ip2}.

The rest of the paper is devoted to the proof of Theorem \ref{it5}. In \S 4 we prove two ``averaging" lemmas. Both are
stated in abstract form and can be read independently. In \S 5 we give a detailed outline of the argument and an
exposition of the main ideas of the proof. In \S 6 we prove the basic tools needed for the proof of Theorem \ref{it5}.
This section is rather technical, and the reader is advised to gain first some familiarity with the general strategy
of the proof before studying this section in detail. Finally, the proof of Theorem \ref{it5} is completed in \S 7. 


\section{Background material}

\numberwithin{equation}{section}

By $\nn=\{0,1,2,...\}$ we shall denote the natural numbers. For every integer $n\meg 1$ we set $[n]=\{1,...,n\}$. 
The cardinality of a set $X$ will be denoted by $|X|$ while its powerset will be denoted by $\mathcal{P}(X)$. 
If $X$ is a nonempty finite set, then by $\ave_{x\in X}$ we shall denote the average $\frac{1}{|X|} \sum_{x\in X}$. 
For every function $f:\nn\to\nn$ and every $k\in\nn$ by $f^{(k)}:\nn\to\nn$ we shall denote the $k$-th iteration of $f$
defined recursively by the rule $f^{(0)}(n)=n$ and $f^{(k+1)}(n)=f\big(f^{(k)}(n)\big)$ for every $n\in\nn$.

\subsection{Trees}

By the term \textit{tree} we mean a nonempty partially ordered set $(T,<)$ such that for every $t\in T$ the set
$\{s\in T: s<t\}$ is linearly ordered under $<$ and finite. The cardinality of this set is defined to be the \textit{length}
of $t$ in $T$ and will be denoted by $\ell_T(t)$. For every $n\in\nn$ the \textit{$n$-level} of $T$, denoted by $T(n)$,
is defined to be the set $\{t\in T: \ell_T(t)=n\}$. The \textit{height} of $T$, denoted by $h(T)$, is defined as follows. 
If there exists $k\in\nn$ with $T(k)=\varnothing$, then set $h(T)=\max\{n\in\nn: T(n)\neq\varnothing\}+1$; otherwise,
set $h(T)=\infty$.

For every node $t$ of a tree $T$ the set of \textit{successors} of $t$ in $T$ is defined by
\begin{equation} \label{e21}
\suc_T(t)=\{s\in T: t\mik s\}.
\end{equation}
Moreover, let $\immsuc_T(t)=\{s\in T: t\mik s \text{ and } \ell_T(s)=\ell_T(t)+1\}$. The set $\immsuc_T(t)$ is called the set
of \textit{immediate successors} of $t$ in $T$. A node $t\in T$ is said to be \textit{maximal} if the set $\immsuc_T(t)$ is empty.

A \textit{subtree} $S$ of a tree $(T,<)$ is a nonempty subset of $T$ viewed as a tree equipped with the induced partial ordering. 
For every $n\in\nn$ with $n<h(T)$ we set
\begin{equation} \label{e24}
T\upharpoonright n= T(0)\cup ... \cup T(n).
\end{equation}
Notice that $h(T\upharpoonright n)=n+1$. An \textit{initial subtree} of $T$ is a subtree of $T$ of the form $T\upharpoonright n$
for some $n\in\nn$.

A tree $T$ is said to be \textit{balanced} if all maximal chains of $T$ have the same cardinality. It is said to be
\textit{uniquely rooted} if $|T(0)|=1$; the \textit{root} of a uniquely rooted tree $T$ is defined to be the node $T(0)$.

\subsection{Vector trees}

A \textit{vector tree} $\bfct$ is a nonempty finite sequence of trees having common height; this common height is defined
to be the \textit{height} of $\bfct$ and will be denoted by $h(\bfct)$.

For every vector tree $\bfct=(T_1, ...,T_d)$ and every $n\in\nn$ with $n< h(\bfct)$ we set
\begin{equation} \label{e25}
\bfct\upharpoonright n=(T_1\upharpoonright n, ..., T_d\upharpoonright n).
\end{equation}
A vector tree of this form is called a \textit{vector initial subtree} of $\bfct$. Also let
\begin{equation} \label{e26}
\bfct(n)=\big(T_1(n),...,T_d(n)\big) \text{ and } \otimes\bfct(n)= T_1(n)\times ...\times T_d(n).
\end{equation}
The \textit{level product of $\bfct$}, denoted by $\otimes\bfct$, is defined to be the set 
\begin{equation} \label{e28}
\bigcup_{n< h(\bfct)} \otimes\bfct(n).
\end{equation}
For every $\bft=(t_1,...,t_d)\in\otimes\bfct$ we set
\begin{equation} \label{e29}
\suc_{\bfct}(\bft)=\big( \suc_{T_1}(t_1), ..., \suc_{T_d}(t_d)\big).
\end{equation}
Finally, we say that a vector tree $\bfct=(T_1,...,T_d)$ is \textit{uniquely rooted} if for every $i\in [d]$ the tree $T_i$
is uniquely  rooted; the element $\bfct(0)$ is called the \textit{root} of $\bfct$.

\subsection{Strong subtrees and vector strong subtrees}

A subtree $S$ of a uniquely rooted tree $T$ is said to be \textit{strong} provided that: (a) $S$ is uniquely rooted and
 balanced, (b) every level of $S$ is a subset of some level of $T$, and (c) for every non-maximal node $s\in S$ and
every $t\in\immsuc_T(s)$ there exists a unique node $s'\in\immsuc_S(s)$ such that $t\mik s'$. The \textit{level set}
of a strong subtree $S$ of $T$ is defined to be the set
\begin{equation} \label{e210}
L_T(S)=\{ m\in\nn: \text{exists } n<h(S) \text{ with } S(n)\subseteq T(m)\}.
\end{equation}
The concept of a strong subtree is naturally extended to vector trees. Specifically, a \textit{vector strong subtree}
of a uniquely rooted vector tree $\bfct=(T_1,...,T_d)$ is a vector tree $\bfcs=(S_1,...,S_d)$ such that $S_i$ is a strong
subtree of $T_i$ for every $i\in [d]$ and $L_{T_1}(S_1)= ...= L_{T_d}(S_d)$.

\subsection{Homogeneous trees and vector homogeneous trees}

Let $b\in\nn$ with $b\meg 2$. By $b^{<\nn}$ we shall denote the set of all finite sequences having values in $\{0,...,b-1\}$. 
The empty sequence is denoted by $\varnothing$ and is included in $b^{<\nn}$. We view $b^{<\nn}$ as a tree equipped
with the (strict) partial order $\sqsubset$ of end-extension. Notice that $b^{<\nn}$ is a homogeneous tree with
branching number $b$. If $n\meg 1$, then $b^{<n}$ stands for the initial subtree of $b^{<\nn}$ of height $n$. 
For every $t,s\in b^{<\nn}$ by $t^{\con}s$ we shall denote the concatenation of $t$ and $s$.

For technical reasons we will not work with abstract homogeneous trees but with a concrete subclass. Observe that
all homogeneous trees with the same branching number are pairwise isomorphic, and so, such a restriction will have
no effect in the generality of our results.
\medskip

\noindent \textbf{Convention.} \textit{In the rest of the paper by the term ``homogeneous tree" (respectively, 
``finite homogeneous tree") we will always mean a strong subtree of $b^{<\nn}$ of infinite (respectively, finite)
height for some integer $b\meg 2$. For every, possibly finite, homogeneous tree $T$ by $b_T$ we shall denote the
branching number of $T$. We follow the same conventions for vector trees. In particular, by the term 
``vector homogeneous tree" (respectively, ``finite vector homogeneous tree") we will always mean a vector strong
subtree of $(b_1^{<\nn},...,b_d^{<\nn})$ of infinite (respectively, finite) height for some integers $b_1,...,b_d$
with $b_i\meg 2$ for all $i\in [d]$. For every, possibly finite, vector homogeneous tree $\bfct=(T_1,...,T_d)$
we set $b_{\bfct}=(b_{T_1},...,b_{T_d})$.}
\medskip

\noindent The above convention enables us to effectively enumerate the set of immediate successors of a given
non-maximal node of a, possibly finite, homogeneous tree $T$. Specifically, for every non-maximal $t\in T$ and
every $p\in\{0,...,b_T-1\}$ let
\begin{equation} \label{e211}
t^{\con_T}\!p= \immsuc_T(t)\cap \suc_{b_T^{<\nn}}(t^{\con}p).
\end{equation}
Notice that $\immsuc_T(t)=\big\{ t^{\con_T}\!p: p\in\{0,...,b_T-1\}\big\}$.

\subsection{Canonical isomorphisms and vector canonical isomorphisms}

Let $T$ and $S$ be two, possibly finite, homogeneous trees with the same branching number and the same height.
The \textit{canonical isomorphism} between $T$ and $S$ is defined to be the unique bijection $\ci:T\to S$ satisfying:
(a) $\ell_T(t)= \ell_S\big(\mathrm{I}(t)\big)$ for every $t\in T$, and (b) $\ci(t^{\con_T}\!p)=\ci(t)^{\con_S}\!p$
for every non-maximal $t\in T$ and every $p\in \{0,...,b_T-1\}$. Observe that if $R$ is a strong subtree of $T$, 
then the image $\ci(R)$ of $R$ under the canonical isomorphism is a strong subtree of $S$ and satisfies $L_T(R)=L_S\big(\ci(R)\big)$.

Respectively, let $\bfct=(T_1,...,T_d)$ and $\bfcs=(S_1,...,S_d)$ be two, possibly finite, vector homogeneous trees
with $b_\bfct=b_\bfcs$ and $h(\bfct)=h(\bfcs)$. For every $i\in [d]$ let $\ci_i$ be the canonical isomorphism between
$T_i$ and $S_i$. The \textit{vector canonical isomorphism} between $\otimes\bfct$ and $\otimes\bfcs$ is the map 
$\bfci: \otimes\bfct\to\otimes\bfcs$ defined by the rule
\begin{equation} \label{e213}
\bfci\big((t_1,...,t_d)\big)=\big(\ci_1(t_1), ..., \ci_d(t_d)\big).
\end{equation}
Notice that the vector canonical isomorphism $\bfci$ is a bijection.

\subsection{Milliken's Theorem}

For every, possibly finite, vector homogeneous tree $\bfct$ and every integer $1\mik k\mik h(\bfct)$ by $\strong_k(\bfct)$
we shall denote the set of all vector strong subtrees of $\bfct$ of height $k$. Moreover, for every $\ell\in \{0,...,k-1\}$ we set
\begin{equation} \label{emilliken1}
\strong^\ell_k(\bfct)=\{\bfcs\in\strong_k(\bfct): \bfcs\upharpoonright \ell=\bfct\upharpoonright\ell\}.
\end{equation}
If $\bfct$ is of infinite height, then $\strong_{\infty}(\bfct)$ stands for the set of all vector strong subtrees of $\bfct$
of infinite height; the set  $\strong^\ell_{\infty}(\bfct)$ is analogously defined. We will need the following elementary fact.
\begin{fact} \label{f21}
Let $d\in\nn$ with $d\meg 1$ and $b_1,...,b_d\in\nn$ with $b_i\meg 2$ for all $i\in [d]$. Let $m\in\nn$ and set
\begin{equation} \label{e214}
q(b_1,...,b_d,m)=\frac{\big( \prod_{i=1}^d b_i^{b_i} \big)^{m+1} -\big( \prod_{i=1}^d b_i \big)^{m+1}}{\prod_{i=1}^d b_i^{b_i}-\prod_{i=1}^d b_i}.
\end{equation}
Also let $\bfct$ be a finite vector homogeneous tree with $b_{\bfct}=(b_1,...,b_d)$ and $h(\bfct)\meg m+2$ and set
\begin{equation} \label{e215}
\strong_2(\bfct,m+1)=\big\{ \bfcf\in\strong_2(\bfct): \otimes\bfcf(1)\subseteq \otimes\bfct(m+1) \big\}
\end{equation}
Then the cardinality of the set $\strong_2(\bfct,m+1)$ is $q(b_1,...,b_d,m)$. In particular,
\begin{equation} \label{emilliken2}
|\strong_2(\bfct)|=\sum_{m=0}^{h(\bfct)-2} q(b_1,...,b_d,m)
\end{equation}
\end{fact}
Notice that if $\bfct=(T_1,...,T_d)$ is a vector homogeneous tree, then the set $\strong_\infty(\bfct)$ is a $G_\delta$
(hence Polish) subspace of $\mathcal{P}(T_1\times...\times T_d)$. The same remark, of course, applies to the set 
$\strong^\ell_{\infty}(\bfct)$ for every $\ell\in\nn$. The following partition result is due to K. Milliken (see \cite[Theorem 2.1]{Mi2}).
\begin{thm} \label{infinite milliken}
Let $\bfct$ be a vector homogeneous tree and $\mathcal{C}\subseteq \strong_{\infty}(\bfct)$ be Borel. Then there
exists $\bfcs\in\strong_{\infty}(\bfct)$ such that either $\strong_{\infty}(\bfcs)\subseteq \mathcal{C}$ or 
$\strong_{\infty}(\bfcs)\cap \mathcal{C}=\varnothing$. Moreover, if $\ell\in\nn$ and   
$\mathcal{B}\subseteq \strong^\ell_{\infty}(\bfct)$ is Borel, then there exists $\bfcs\in\strong^\ell_{\infty}(\bfct)$
such that either $\strong^\ell_{\infty}(\bfcs)\subseteq \mathcal{B}$ or $\strong^\ell_{\infty}(\bfcs)\cap \mathcal{B}=\varnothing$.

In particular, for every integer $k\meg 1$ and every finite partition of $\strong_k(\bfct)$ there exists $\bfcs\in\strong_{\infty}(\bfct)$
such that the set $\strong_k(\bfcs)$ is monochromatic. Respectively, for every $\ell\in \{0,...,k-1\}$ and every finite
partition of $\strong^\ell_k(\bfct)$ there exists $\bfcs\in\strong^{\ell}_{\infty}(\bfct)$ such that the set
$\strong^{\ell}_k(\bfcs)$ is monochromatic.
\end{thm}
Theorem \ref{infinite milliken} has a finite version which is also due to K. Milliken (see \cite{Mi1}).
\begin{thm} \label{t22}
For every integer $d\meg 1$, every $b_1,...,b_d\in\nn$ with $b_i\meg 2$ for all $i\in [d]$, every pair of integers
$m\meg k\meg 1$ and every integer $r\meg 2$ there exists an integer $M$ with the following property. For every finite
vector homogeneous tree $\bfct$ with $b_{\bfct}=(b_1,...,b_d)$ and of height at least $M$ and every $r$-coloring
of the set $\strong_k(\bfct)$ there exists $\bfcs\in\strong_m(\bfct)$ such that the set $\strong_k(\bfcs)$ is monochromatic. 
The least integer $M$ with this property will be denoted by $\mil(b_1,...,b_d|m,k,r)$.
\end{thm}
The original proof of Theorem \ref{t22} did not provide quantitative information on the numbers $\mil(b_1,...,b_d|m,k,r)$. 
However, an analysis of the finite version of Milliken's Theorem has been carried out recently by M. Soki\'{c} yielding explicit
and reasonable upper bounds. In particular, it was shown in \cite{So} (see, also, \cite[\S 10]{DKT2}) that for every integer
$k\meg 1$ there exists a primitive recursive function $\phi_k:\nn^3\to \nn$ belonging to the class $\mathcal{E}^{5+k}$ of
Grzegorczyk's hierarchy such that for every integer $d\meg 1$, every $b_1,...,b_d\in\nn$ with $b_i\meg 2$ for all $i\in [d]$,
every integer $m\meg k$ and every integer $r\meg 2$ we have 
\begin{equation} \label{e216}
\mil(b_1,...,b_d|m,k,r) \mik \phi_k\Big( \prod_{i=1}^d b_i^{b_i},m,r\Big).
\end{equation}
We close this subsection with the following consequence of Theorem \ref{t22}.
\begin{cor} \label{c24}
Let $d\in\nn$ with $d\meg 1$ and $b_1,...,b_d\in\nn$ with $b_i\meg 2$ for all $i\in [d]$. Also let $m,k,r\in\nn$ with $m\meg k\meg 1$
and $r\meg 2$. If $\bfct$ is finite vector homogeneous tree with $b_{\bfct}=(b_1,...,b_d)$ and
\begin{equation} \label{e218}
h(\bfct)\meg \mil\big(\underbrace{b_1,...,b_1}_{b_1-\mathrm{times}}, ..., \underbrace{b_d,...,b_d}_{b_d-\mathrm{times}}|m,k,r\big)+1
\end{equation}
then for every $r$-coloring of $\strong_{k+1}^0(\bfct)$ there exists $\bfcr\in\strong_{m+1}^0(\bfct)$ such that the set
$\strong_{k+1}^0(\bfcr)$ is monochromatic.
\end{cor}

\subsection{The ``uniform" version of the finite density Halpern--L\"{a}uchli Theorem}

We will need the following result (see \cite[Theorem 3]{DKT2}).
\begin{thm} \label{uniform dhl}
For every integer $d\meg 1$, every $b_1,...,b_d\in\nn$ with $b_i\meg 2$ for all $i\in [d]$, every integer $k\meg 1$ and 
every real $0<\ee\mik 1$ there exists an integer $N$ with the following property. If $\bfct=(T_1,...,T_d)$ is a vector
homogeneous tree with $b_{\bfct}=(b_1,...,b_d)$, $L$ is a finite subset of $\nn$ of cardinality at least $N$ and $D$ is
a subset of the level product of $\bfct$ satisfying $|D\cap \otimes\bfct(n)| \meg \ee |\!\otimes\bfct(n)|$ for every $n\in L$, 
then there exists a vector strong subtree $\bfcs$ of $\bfct$ of height $k$ such that the level product of $\bfcs$ is contained
in $D$. The least integer $N$ with this property will be denoted by $\udhl(b_1,...,b_d|k,\ee)$.
\end{thm}
We point out that the case ``$d=1$" of Theorem \ref{uniform dhl} is due to J. Pach, J. Solymosi and G. Tardos \cite{PST}
who obtained the upper bound $\udhl(b|k,\ee)=O_{b,\ee}(k)$. The proof of the higher-dimensional case, given in \cite{DKT2},
is effective and gives explicit upper bounds for the numbers $\udhl(b_1,...,b_d|k,\ee)$. These upper bounds, however, have
an Ackermann-type dependence with respect to the ``dimension" $d$.

\subsection{Embedding finite subsets of the level product of a vector homogeneous tree}

We will need the following embedding result.
\begin{prop} \label{ap-p1}
Let $d\in\nn$ with $d\meg 1$ and $\bfct=(T_1,...,T_d)$ be a vector homogeneous tree. Then for every integer $n\meg 1$ and
every subset $F$ of the level product $\otimes\bfct$ of $\bfct$ of cardinality $n$ there exists a vector strong subtree 
$\bfcs$ of $\bfct$ of height $d(2n-1)$ such that the set $F$ is contained in the level product $\otimes\bfcs$ of $\bfcs$.
\end{prop}
\begin{proof}
Before we proceed to the details we need, first, to introduce some pieces of notation. Specifically, let $T$ be a homogeneous
tree and $A$ be a nonempty subset of $T$. The \textit{level set} $L_T(A)$ of $A$ in $T$ is defined to be the set 
$\{n\in\nn: T(n)\cap A\neq\varnothing\}$. Now fix a subset $F$ of $\otimes\bfct$ with $|F|=n$ and for every $i\in [d]$ let
\begin{equation} \label{ap-e6}
F_i=\{t\in T_i: \text{there exists } (t_1,...,t_d)\in F \text{ with } t=t_i\}
\end{equation}
be the projection of the set $F$ on the tree $T_i$. Notice that there exists a subset $L$ of $\nn$ with $|L|\mik n$ such that
$L_{T_i}(F_i)=L$ for every $i\in [d]$. Moreover,  $|F_i|\mik |F|$ for every $i\in [d]$. Finally, observe that
\begin{equation} \label{ap-e7}
F\subseteq \bigcup_{n\in L} \big(T_1(n)\cap F_1\big)\times ... \times \big(T_d(n)\cap F_d\big).
\end{equation}
By \cite[Corollary 38]{DKT1}, for every $i\in [d]$ there exists a strong subtree $R_i$ of $T_i$ of height $2n-1$
such that $F_i\subseteq R_i$. Since
\begin{equation}
|L_{T_1}(R_1) \cup ... \cup L_{T_d}(R_d)| \mik d(2n-1)
\end{equation}
we may select a subset $M$ of $\nn$ with $L_{T_1}(R_1)\cup ...\cup L_{T_d}(R_d)\subseteq M$ and $|M|=d(2n-1)$. 
Next, for every $i\in [d]$ we select a strong subtree $S_i$ of $T_i$ with $L_{T_i}(S_i)=M$ and $R_i\subseteq S_i$. 
By (\ref{ap-e7}), we see that the vector strong subtree $\bfcs=(S_1,...,S_d)$ of $\bfct$ is as desired.
\end{proof}
We point out that the bound of the height of the vector strong subtree obtained by Proposition \ref{ap-p1} is not
optimal. Actually, by appropriately modifying the proof of \cite[Corollary 38]{DKT1} and arguing as above, it is
possible to show that the desired vector strong subtree can be chosen to have height at most $n+d(n-1)$, an upper
bound which is easily seen to be sharp. Such an improvement, however, has only minor effect on the estimates obtained
in the rest of the paper, and so, we prefer not to bother the reader with it.

\subsection{Probabilistic preliminaries}

For every $0<\theta<\ee\mik 1$ and every integer $k\meg 2$ we set
\begin{equation} \label{2pr-e1}
\Sigma(\theta,\ee,k)=\Big\lceil \frac{k(k-1)}{2(\ee^k-\theta^k)}\Big\rceil.
\end{equation}
We will need the following well-known fact. We give the proof for completeness.
\begin{lem} \label{correlation}
Let $0< \theta< \ee\mik 1$ and $k,N\in\nn$ with $k\meg 2$ and $N\meg\Sigma(\theta,\ee,k)$. Also let $(A_i)_{i=1}^N$
be a family of measurable events in a probability space $(\Omega,\Sigma,\mu)$ such that $\mu(A_i)\meg\ee$ for all
$i\in [N]$. Then there exists a subset $F$ of $[N]$ of cardinality $k$ such that
\begin{equation} \label{2pr-e2}
\mu\Big( \bigcap_{i\in F} A_i\Big) \meg \theta^k.
\end{equation}
\end{lem}
\begin{proof}
Let $\mathcal{A}$ be the set of all functions $\sigma:[k]\to [N]$ and $\mathcal{B}=\{\sigma\in\mathcal{A}:\sigma \text{ is 1-1}\}$. 
Notice that
\begin{equation} \label{2pr-e3}
|\mathcal{A}\setminus\mathcal{B}|\mik \frac{k(k-1)}{2} N^{k-1}.
\end{equation}
By our assumptions and Jensen's inequality, we see that
\begin{eqnarray} \label{2pr-e4}
\ee^k N^k & \mik & \Big( \int \sum_{i=1}^{N} \mathbf{1}_{A_i} d\mu \Big)^k \mik \int\Big(\sum_{i=1}^{N} 
\mathbf{1}_{A_i}\Big)^k d\mu \\
& = & \int \sum_{\sigma\in\mathcal{A}} \prod_{i=1}^{k} \mathbf{1}_{A_{\sigma(i)}} d\mu = 
\sum_{\sigma\in\mathcal{A}} \mu\Big( \bigcap_{i=1}^{k} A_{\sigma(i)}\Big) \nonumber \\
& = & \sum_{\sigma\in\mathcal{A}\setminus \mathcal{B}} \mu\Big( \bigcap_{i=1}^{k} A_{\sigma(i)}\Big) +
\sum_{\sigma\in\mathcal{B}} \mu\Big( \bigcap_{i=1}^{k} A_{\sigma(i)}\Big) \nonumber \\
& \stackrel{(\ref{2pr-e3})}{\mik} & \frac{k(k-1)}{2} N^{k-1} + \sum_{\sigma\in\mathcal{B}}
\mu\Big( \bigcap_{i=1}^{k} A_{\sigma(i)}\Big). \nonumber
\end{eqnarray}
Since $N\meg \Sigma(\theta,\ee,k)$ we get
\begin{equation} \label{2pr-e5}
\ave_{\sigma\in\mathcal{B}} \ \mu \Big(\bigcap_{i=1}^k A_{\sigma(i)}\Big) \meg 
\frac{1}{N^k} \sum_{\sigma\in\mathcal{B}} \mu \Big(\bigcap_{i=1}^k A_{\sigma(i)}\Big)
\stackrel{(\ref{2pr-e4})}{\meg} \ee^k- \frac{1}{N}\cdot \frac{k(k-1)}{2}\meg \theta^k
\end{equation}
and the proof is completed.
\end{proof}
We will also need three elementary variants of Markov's inequality. We isolate them, below, for the convenience of the reader.
\begin{fact} \label{markov-f1}
Let $0<\alpha\mik 1$, $N\in\nn$ with $N\meg 1$ and $a_1,...,a_N$ in $[0,1]$. Assume that $\mathbb{E}_{i\in [N]} a_i \meg \alpha$. 
Then for every $0<\gamma<\alpha$ we have $|\{ i\in [N]: a_i\meg \alpha-\gamma\}| \meg \gamma N$.
\end{fact}
\begin{fact} \label{markov-f2}
Let $0<\alpha\mik 1$, $N\in\nn$ with $N\meg 1$ and $a_1,...,a_N$ in $[0,1]$ such that $\mathbb{E}_{i\in [N]} a_i \meg \alpha$.
Also let $\gamma>0$ and assume that $|\{ i\in [N]: a_i\meg \alpha+\gamma^2\}| \mik \gamma^3 N$. Then 
$|\{ i\in [N]: a_i\meg \alpha-\gamma \}|\meg (1-\gamma)N$.
\end{fact}
\begin{fact} \label{markov-f3}
Let $N\in\nn$ with $N\meg 1$ and $a_1,...,a_N$ in $[0,+\infty)$. Also let $\theta,\lambda>0$ and assume that
$\mathbb{E}_{i\in [N]} a_i\mik\theta$. Then  $|\{ i\in [N]: a_i\mik \theta\lambda^{-1}\}|\meg (1-\lambda)N$.
\end{fact}


\section{Proof of Theorem \ref{it1}}

\numberwithin{equation}{section}

This section is devoted to the proof of Theorem \ref{it1} stated in the introduction. It is organized as follows. 
In \S 3.1 we define certain numerical parameters. In \S 3.2 we give the proof of Proposition \ref{ip2} while
 in \S 3.3 we present some of its consequences. As we have already mentioned, Proposition \ref{ip2} is the main
``pigeonhole principle" in the proof of Theorem \ref{it1}. It is used in \S 3.4 where we show that we can always
pass to a vector strong subtree $\bfcr$ for which we have significant control for the events indexed by the level
product of every initial subtree of $\bfcr$. The proof of Theorem \ref{it1} is completed in \S 3.5. Finally, 
in \S 3.6 we give a combinatorial application.

\subsection{Defining certain numerical parameters}

Let $d\in\nn$ with $d\meg 1$ and $b_1,...,b_d\in\nn$ with $b_i\meg 2$ for every $i\in [d]$. Also let $0<\ee\mik 1$. We set
\begin{equation} \label{3e1}
\mathrm{Cor}(b_1,...,b_d|\ee)=\Sigma\big( \ee/4,\ee/2, \udhl(b_1,...,b_d|2,\ee/2)\big)
\end{equation}
where $\Sigma(\theta,\ee,k)$ is defined in (\ref{2pr-e1}). Next we set
\begin{equation} \label{3e3}
Q(b_1,...,b_d|\ee)= \sum_{m=0}^{\mathrm{Cor}(b_1,...,b_d|\ee)-2}
\frac{\big( \prod_{i=1}^d b_i^{b_i} \big)^{m+1} -\big( \prod_{i=1}^d b_i \big)^{m+1}}{\prod_{i=1}^d b_i^{b_i} -\prod_{i=1}^d b_i}
\end{equation}
and we define
\begin{equation} \label{3e4}
\xi(b_1,...,b_d|\ee)= \frac{\big(\frac{\ee}{4}\big)^{\udhl(b_1,...,b_d|2,\ee/2)}}{Q(b_1,...,b_d|\ee)}.
\end{equation}
Recursively we define a sequence of positive reals by the rule
\begin{equation} \label{3e5}
\left\{ \begin{array} {l} \xi_1(b_1,...,b_d|\ee)=\ee, \\ \xi_{k+1}(b_1,...,b_d|\ee)=
\xi\big(b_1,...,b_d|\xi_k(b_1,...,b_d|\ee)\big) \end{array}  \right.
\end{equation}
and for every integer $n\meg 1$ we set
\begin{equation} \label{3e6}
c(b_1,...,b_d|n,\ee)= \xi_{d(2n-1)}(b_1,...,b_d|\ee).
\end{equation}

\subsection{Proof of Proposition \ref{ip2}}

Clearly we may assume that
\[ h(\bfct)=\mathrm{Cor}(b_1,...,b_d|\ee).\]
For every $n\in\{0,...,\mathrm{Cor}(b_1,...,b_d|\ee)-1\}$ we select $\Omega_n\in\Sigma$ with $\mu(\Omega_n)\meg\ee/2$
such that for every $\omega\in\Omega_n$ we have
\begin{equation} \label{3e7}
|\{\bft\in\otimes\bfct(n): \omega\in A_{\bft}\}|\meg (\ee/2) |\!\otimes\bfct(n)|.
\end{equation}
By (\ref{3e1}) and Lemma \ref{correlation}, there exists a subset $L$ of $\{0,...,\mathrm{Cor}(b_1,...,b_d|\ee)-1\}$ with
\begin{equation} \label{3e8}
|L|= \udhl(b_1,...,b_d|2,\ee/2)
\end{equation}
such that, setting $G=\bigcap_{n\in L} \Omega_n$, we have
\begin{equation} \label{3e10}
\mu(G)\meg \Big(\frac{\ee}{4}\Big)^{\udhl(b_1,...,b_d|2,\ee/2)}.
\end{equation}
Let $\omega\in G$ be arbitrary and set $D_\omega=\{\bft\in\otimes\bfct: \omega\in A_{\bft}\}$. By (\ref{3e7}), 
for every $n\in L$ we have $|D_\omega\cap \otimes\bfct(n)|\meg (\ee/2) |\!\otimes\bfct(n)|$. Hence, by (\ref{3e8})
and Theorem \ref{uniform dhl}, there exists $\bfcf_\omega\in\strong_2(\bfct)$ with $\otimes\bfcf_\omega\subseteq D_{\omega}$.
Now observe that, by Fact \ref{f21} and (\ref{3e3}), we have $|\strong_2(\bfct)|=Q(b_1,...,b_d|\ee)$. Thus, invoking (\ref{3e4})
and (\ref{3e10}), we see that there exist $\bfcf\in\strong_2(\bfct)$ and $G'\in\Sigma$ with $\mu(G')\meg\xi(b_1,...,b_d|\ee)$
such that $G'\subseteq G$ and $\bfcf_\omega=\bfcf$ for every $\omega\in G'$. Therefore,
\begin{equation} \label{3e13}
\mu\Big( \bigcap_{\bft\in\otimes\bfcf} A_{\bft}\Big) \meg \mu(G') \meg \xi(b_1,...,b_d|\ee)
\end{equation}
and the proof is completed.

\subsection{Consequences}

Proposition \ref{ip2} will be used in the following form.
\begin{cor} \label{3c1}
Let $d\in\nn$ with $d\meg 1$ and $b_1,...,b_d\in\nn$ with $b_i\meg 2$ for all $i\in [d]$. Also let $k\in\nn$ and $0<\ee\mik 1$.
Assume that $\bfcs=(S_1,...,S_d)$ is a finite vector homogeneous tree with $b_{\bfcs}=(b_1,...,b_d)$ and
\begin{equation} \label{3e14}
h(\bfcs)\meg (k+1)+\mathrm{Cor}(b_1,...,b_d|\ee)
\end{equation}
and $\{A_\bfs:\bfs\in\otimes\bfcs\}$ is a family of measurable events in a probability space $(\Omega,\Sigma,\mu)$
such that for every $\bfcp\in\strong^k_{k+2}(\bfcs)$ we have
\begin{equation} \label{3e15}
\mu\Big( \bigcap_{\bfs\in\otimes\bfcp} A_{\bfs}\Big)\meg \ee.
\end{equation}
Then there exists $\bfcg\in\strong^k_{k+3}(\bfcs)$ such that
\begin{equation} \label{3e15new}
\mu\Big( \bigcap_{\bfs\in\otimes\bfcg} A_{\bfs}\Big)\meg \xi(b_1,...,b_d|\ee).
\end{equation}
\end{cor}
\begin{proof}
We will reduce the proof to Proposition \ref{ip2} using the notion of vector canonical isomorphism. 
To this end we need, first, to do some preparatory work. We set $h=h(\bfcs)-(k+1)$ and $\bfcb=(b_1^{<h},..., b_d^{<h})$.

Let $\bfs\in\otimes \bfcs(k+1)$ be arbitrary and write $\suc_{\bfcs}(\bfs)=(S_1^\bfs,...,S_d^\bfs)$. For every
$i\in [d]$ the finite homogeneous trees $b_i^{<h}$ and $S_i^\bfs$ have the same branching number and the same height.
Therefore, as we described in \S 2.5, we may consider the canonical isomorphism $\ci_i^\bfs:b_i^{<h}\to S_i^\bfs$. 
The same remarks apply to the finite vector homogeneous trees $\bfcb$ and $\suc_{\bfcs}(\bfs)$. Thus we may also
consider the vector canonical isomorphism  $\bfci_\bfs:\otimes\bfcb\to\otimes\suc_{\bfcs}(\bfs)$ given by the family
of maps $\big\{\ci_i^\bfs:i\in [d]\big\}$ via formula (\ref{e213}). Notice that for every $\bfs,\bft\in\otimes\bfcs(k+1)$
and every $i\in[d]$ if $S_i^\bfs=S_i^\bft$ (that is, if the finite sequences $\bfs$ and $\bft$ agree on the $i$-th coordinate),
then the maps $\ci_i^\bfs$ and $\ci_i^\bft$ are identical. This coherence property yields the following fact.
\begin{fact} \label{3f2}
For every $\ell\in \{1,...,h\}$ and every $\bfcu\in\strong_{\ell}(\bfcb)$ there exists a unique
$\bfcg_{\bfcu}\in\strong^k_{(k+1)+\ell}(\bfcs)$ such that
\begin{equation} \label{3e16}
\otimes\bfcg_{\bfcu}= (\otimes\bfcs\upharpoonright k) \cup \{ \bfci_{\bfs}(\bfu): \bfs\in\otimes\bfcs(k+1) 
\text{ and } \bfu\in\otimes\bfcu\}.
\end{equation}
\end{fact}
After this preliminary discussion we are ready to proceed to the proof. For every $\bfu\in\otimes\bfcb$ we set
\begin{equation} \label{3e17}
B_{\bfu}=\bigcap_{\bft\in\otimes\bfcs\upharpoonright k} A_\bft \cap \bigcap_{\bfs\in\otimes\bfcs(k+1)} A_{\bfci_{\bfs}(\bfu)}.
\end{equation}
First we claim that $\mu(B_{\bfu})\meg\ee$ for every $\bfu\in\otimes\bfcb$. Indeed, let $\bfu\in\otimes\bfcb$
be arbitrary and notice that $\bfu\in\strong_1(\bfcb)$. Hence, by Fact \ref{3f2} and the definition of the set
$B_{\bfu}$ in (\ref{3e17}) above, there exists $\bfcg_{\bfu}\in\strong^k_{k+2}(\bfcs)$ such that
\begin{equation} \label{3e18}
B_{\bfu}= \bigcap_{\bft\in\otimes \bfcg_{\bfu}} A_{\bft}
\end{equation}
and the claim follows from our hypotheses. Next we observe that 
\begin{equation} \label{3e19}
h(\bfcb)=h(\bfcs)-(k+1) \stackrel{(\ref{3e14})}{\meg} \mathrm{Cor}(b_1,...,b_d|\ee).
\end{equation}
Therefore, by Proposition \ref{ip2}, there exists $\bfcf\in\strong_2(\bfcb)$ such that
\begin{equation} \label{3e20}
\mu \Big( \bigcap_{\bfu\in\otimes\bfcf} B_{\bfu}\Big) \meg \xi(b_1,...,b_d|\ee).
\end{equation}
Invoking Fact \ref{3f2} and (\ref{3e17}) again, there exists $\bfcg_{\bfcf}\in\strong^k_{k+3}(\bfcs)$ such that
\begin{equation} \label{3e21}
\bigcap_{\bfu\in\otimes\bfcf} B_{\bfu}= \bigcap_{\bft\in\otimes \bfcg_{\bfcf}} A_{\bft}.
\end{equation}
Combining (\ref{3e20}) and (\ref{3e21}), the result follows.
\end{proof}

\subsection{Control of initial subtrees}

This subsection is devoted to the proof of the following lemma which is the last step towards the proof of Theorem \ref{it1}.
\begin{lem} \label{3l3}
Let $d\in\nn$ with $d\meg 1$ and $b_1,...,b_d\in\nn$ with $b_i\meg 2$ for all $i\in [d]$. Also let $0<\ee\mik 1$.
If $\bfct=(T_1,..., T_d)$ is a vector homogeneous tree with $b_{\bfct}=(b_1,...,b_d)$ and $\{A_{\bft}:\bft\in\otimes\bfct\}$
is a family of measurable events in a probability space $(\Omega,\Sigma,\mu)$ satisfying $\mu(A_\bft)\meg\ee$ for every
$\bft\in\otimes\bfct$, then there exists a vector strong subtree $\bfcr$ of $\bfct$ of infinite height such that for
every $k\in\nn$ we have
\begin{equation} \label{3e22}
\mu\Big( \bigcap_{\bfr\in\otimes \bfcr\upharpoonright k} A_\bfr\Big) \meg \xi_{k+1}(b_1,...,b_d|\ee)
\end{equation}
where $\xi_{k+1}(b_1,...,b_d|\ee)$ is defined in \eqref{3e5}.
\end{lem}
\begin{proof}
Recursively, we shall construct a sequence $(\bfcr_k)$ of vector strong subtrees of $\bfct$ of infinite height
such that for every $k\in\nn$ the following conditions are satisfied.
\begin{enumerate}
\item[(C1)] We have $\bfcr_{k+1}\upharpoonright k=\bfcr_k\upharpoonright k$.
\item[(C2)] For every $\bfcg\in\strong^k_{k+2}(\bfcr_k)$ we have
\begin{equation} \label{3e23}
\mu\Big( \bigcap_{\bft\in\otimes\bfcg} A_{\bft}\Big) \meg \xi_{k+2}(b_1,...,b_d|\ee).
\end{equation}
\end{enumerate}
To select the tree $\bfcr_0$ we argue as follows. Let
\begin{equation} \label{3e23new}
\mathcal{F}=\Big\{ \bfcf\in\strong_2(\bfct): \mu\Big( \bigcap_{\bft\in\otimes\bfcf} A_{\bft}\Big) \meg \xi_2(b_1,...,b_d|\ee)\Big\}.
\end{equation}
Applying Theorem \ref{infinite milliken}, we may find $\bfcs\in\strong_{\infty}(\bfct)$ such that either 
$\strong_2(\bfcs)\subseteq \mathcal{F}$ or $\strong_2(\bfcs)\cap \mathcal{F}=\varnothing$. Since $\mu(A_\bfs)\meg\ee$
for every $\bfs\in\otimes\bfcs$, by Proposition \ref{ip2}, we see that $\strong_2(\bfcs)\cap\mathcal{F}\neq\varnothing$. 
Therefore, $\strong_2(\bfcs)\subseteq \mathcal{F}$. We set ``$\bfcr_0=\bfcs$" and we observe that condition (C2) is satisfied.
Since condition (C1) is meaningless for ``$k=0$", the first step of the recursive selection is completed.

Let $k\in\nn$ and assume that the trees $\bfcr_0,..., \bfcr_k$ have been selected so that conditions (C1) and (C2) are satisfied. We set
\begin{equation} \label{3e24}
\mathcal{G}=\Big\{ \bfcg\in\strong^k_{k+3}(\bfcr_k): \mu\Big( \bigcap_{\bft\in\otimes\bfcg} A_{\bft}\Big) \meg \xi_{k+3}(b_1,...,b_d|\ee)\Big\}.
\end{equation}
Arguing as above and using Theorem \ref{infinite milliken} and Corollary \ref{3c1}, we see that there exists
$\bfcs\in\strong^k_{\infty}(\bfcr_k)$ such that $\strong^k_{k+3}(\bfcs)\subseteq \mathcal{G}$. We set ``$\bfcr_{k+1}=\bfcs$"
and we observe that with this choice conditions (C1) and (C2) are satisfied. The recursive selection is thus completed.

Now let $\bfcr$ be the unique vector strong subtree of $\bfct$ such that $\bfcr\upharpoonright k=\bfcr_k\upharpoonright k$
for every $k\in\nn$. Notice that, by condition (C1), $\bfcr$ is well-defined. Let $k\in\nn$ be arbitrary. If $k=0$, then
$\bfcr\upharpoonright 0$ is just the root $\bfcr(0)$ of $\bfcr$; so in this case (\ref{3e22})  follows from our hypotheses.
If $k\meg 1$, then by condition (C1) we see that $\bfcr\upharpoonright k =\bfcr_k\upharpoonright k\in \strong^{k-1}_{k+1}(\bfcr_{k-1})$. 
Therefore, by condition (C2),
\begin{equation} \label{3e25}
\mu\Big( \bigcap_{\bfr\in\otimes \bfcr\upharpoonright k} A_{\bfr}\Big) \meg \xi_{(k-1)+2}(b_1,...,b_d|\ee)=\xi_{k+1}(b_1,...,b_d|\ee)
\end{equation}
and the proof is completed.
\end{proof}

\subsection{Proof of Theorem \ref{it1}}
Let
\begin{equation} \label{3e26}
\mathcal{C}=\Big\{ \bfcr\in\strong_{\infty}(\bfct): \mu\Big(\bigcap_{\bfr\in \otimes\bfcr\upharpoonright k} A_\bfr\Big)\meg
\xi_{k+1}(b_1,...,b_d|\ee) \text{ for every } k\in\nn \Big\}.
\end{equation}
It is easy to see that $\mathcal{C}$ is a closed subset of $\strong_{\infty}(\bfct)$. By Theorem \ref{infinite milliken}
and Lemma \ref{3l3}, there exists $\bfcs\in\strong_{\infty}(\bfct)$ such that $\strong_{\infty}(\bfcs)\subseteq \mathcal{C}$.
The vector strong subtree $\bfcs$ is the desired one. Indeed, let $n\in\nn$ with $n\meg 1$ and $F$ be an arbitrary subset
of $\otimes\bfcs$ of cardinality $n$. By Proposition \ref{ap-p1}, there exists a vector strong subtree $\bfcg$ of $\bfcs$
of height $d(2n-1)$ such that $F\subseteq\otimes\bfcg$. Observe that there exists $\bfcr\in\strong_{\infty}(\bfcs)$ such that,
setting $k=d(2n-1)-1$, we have that $\bfcg=\bfcr\upharpoonright k$. Since $\bfcr\in\strong_{\infty}(\bfcs)\subseteq\mathcal{C}$,
\begin{equation}
\mu\Big(\bigcap_{\bft\in F} A_\bft\Big) \meg \mu\Big(\bigcap_{\bft\in \otimes\bfcg} A_\bft\Big)\meg
\xi_{d(2n-1)}(b_1,...,b_d|\ee) \stackrel{(\ref{3e6})}{=} c(b_1,..., b_d|n,\ee).
\end{equation}
The proof of Theorem \ref{it1} is completed.

\subsection{A combinatorial application: random colorings of strong subtrees}

An old problem of P. Erd\H{o}s and A. Hajnal \cite[page 115]{EH} asked whether given a family $\{A_{\{n,m\}}:\{n,m\}\in [\nn]^2\}$
of measurable events in a probability space $(\Omega,\Sigma,\mu)$ satisfying $\mu(A_{\{n,m\}})\meg \ee>0$ for every $\{n,m\}\in [\nn]^2$
there exists an infinite subset $L=\{n_0< n_1< ...\}$ of $\nn$ such that the set $\bigcap_{i\in\nn} A_{\{n_i,n_{i+1}\}}$ is nonempty. 
This problem is pointing towards obtaining a ``random" version of the classical Ramsey Theorem \cite{Ra} and it was resolved by
D. H. Fremlin and M. Talagrand in \cite{FT} who showed that there exists a critical threshold: if the underlying probability
space is the unit interval with the Lebesgue measure, then such an infinite subset $L$ can be found if and only if $\ee>1/2$.

One can consider a ``tree" version of the Erd\H{o}s--Hajnal problem where doubletons of $\nn$ are replaced with strong
subtrees of height $2$ of a fixed homogeneous tree $T$ and infinite subsets of $\nn$ with strong subtrees of $T$ of infinite
height. Of course, such a question asks if a ``random" version of Milliken's Theorem \cite{Mi1,Mi2} holds true. Since we can
naturally ``code" doubletons of $\nn$ with elements of $\strong_2(T)$ via their level set, we see that in the ``tree" version
one will also face threshold phenomena. Such threshold phenomena, however, do not occur if we restrict our attention to strong
subtrees with a fixed root.
\begin{cor} \label{erdos-hajnal}
Let $b\in\nn$ with $b\meg 2$ and $0<\ee\mik 1$. Also let $T$ be a homogeneous tree with branching number $b$ and
$\{A_S:S\in\strong_2(T)\}$ be a family of measurable events in a probability space $(\Omega,\Sigma,\mu)$ satisfying
$\mu(A_S)\meg\ee$ for every $S\in \strong_2(T)$. Then there exists a strong subtree $R$ of $T$ of infinite height
such that for every integer $n\meg 1$ and every $S_1,...,S_n\in\strong_2(R)$ with $S_1(0)=...=S_n(0)$ we have
\begin{equation} \label{e1-erdos-hajnal}
\mu\Big( \bigcap_{i=1}^n A_{S_i}\Big) \meg c\big(\underbrace{b,...,b}_{b-\mathrm{times}}|n,\ee\big)
\end{equation}
where the constant in \eqref{e1-erdos-hajnal} is as in Theorem \ref{it1}.
\end{cor}
\begin{proof}
Let $Z\in\strong_{\infty}(T)$ be arbitrary and set $z=Z(0)$. Write the set $\immsuc_Z(z)$ in lexicographical increasing
order as $\{z_1\lex ...\lex z_b\}$ and set
\begin{equation}
\bfcs=\big(\suc_Z(z_1),...,\suc_Z(z_b)\big).
\end{equation}
Notice that there exists a natural isomorphism $\Phi:\otimes\bfcs\to\strong^0_2(Z)$ defined by
$\Phi\big((s_1,...,s_b)\big)=\{z\}\cup\{s_1,...,s_b\}$. Using this observation, the result follows
by Theorem \ref{it1} and a standard recursive construction.
\end{proof}


\section{Two ``averaging" lemmas}

\numberwithin{equation}{section}

This section is devoted to the proof of two ``averaging" lemmas. Both are stated in abstract form and concern
the structure of real-valued functions of two variables. They will be applied in \S 6. Before we proceed to the
details we need, first, to define some auxiliary quantities. Specifically, for every $0<\alpha\mik \beta\mik 1$
and every $0<\varrho\mik 1$ we set
\begin{equation} \label{e61}
\gamma_0=\gamma_0(\alpha,\beta,\varrho)=(\beta+\varrho^2-\alpha)^{1/2},
\end{equation}
\begin{equation} \label{e62}
\gamma_1=\gamma_1(\alpha,\beta,\varrho)=(\gamma_0+\gamma_0^2)^{1/2}
\end{equation}
and
\begin{equation} \label{e63}
\gamma_2=\gamma_2(\alpha,\beta,\varrho)=(\gamma_1+\gamma_1^2)^{1/2}.
\end{equation}
We isolate, for future use, the following elementary properties.
\begin{enumerate}
\item[($\mathcal{P}$1)] We have $\beta+\varrho^2=\alpha+\gamma_0^2=(\alpha-\gamma_0)+\gamma_1^2=(\alpha-\gamma_0-\gamma_1)+\gamma_2^2$.
\item[($\mathcal{P}$2)] We have $0<\varrho\mik \gamma_0 <\gamma_1<\gamma_2$.
\item[($\mathcal{P}$3)] If $\gamma_0\mik(\alpha/4)^4$, then $\gamma_1\mik (2 \gamma_0)^{1/2}$, $\gamma_2\mik 2 \gamma_0^{1/4}$
and $\alpha-\gamma_0-\gamma_1-\gamma_2\meg \alpha/4$.
\end{enumerate}
We are ready to state the first main result of this section.
\begin{lem} \label{l62}
Let $0<\alpha\mik\beta\mik 1$ and $0<\varrho\mik 1$ and define $\gamma_0, \gamma_1$ and $\gamma_2$ as in
\eqref{e61}, \eqref{e62} and \eqref{e63} respectively. Assume that
\begin{equation} \label{e64}
\gamma_0\mik \Big( \frac{\alpha}{4}\Big)^4.
\end{equation}
Assume, moreover, that we are given
\begin{enumerate}
\item[(a)] two nonempty finite sets $\mathcal{S}$ and $\mathcal{W}$,
\item[(b)] an integer $h\meg 1$ and a partition $\{\mathcal{S}_0, ...,\mathcal{S}_{h-1}\}$ of $\mathcal{S}$, and
\item[(c)] a function $f:\mathcal{S}\times\mathcal{W}\to [0,1]$ such that
$\mathbb{E}_{w\in\mathcal{W}}\mathbb{E}_{n<h}\mathbb{E}_{s\in \mathcal{S}_n} f(s,w)\meg \alpha$.
\end{enumerate}
Then, either
\begin{enumerate}
\item[(i)] there exist $w_0\in\mathcal{W}$ and $\mathcal{N}_0\subseteq \{0,...,h-1\}$ with $|\mathcal{N}_0|\meg \varrho^3 h$ such that, setting
\begin{equation} \label{4e1new}
\Delta_0=\{s\in\mathcal{S}:f(s,w_0)\meg\beta+\varrho^2/2\},
\end{equation}
we have $|\Delta_0\cap\mathcal{S}_n|\meg \varrho^3|\mathcal{S}_n|$ for every $n\in\mathcal{N}_0$,  or
\item[(ii)] there exists $\mathcal{W}^*\subseteq\mathcal{W}$ with $|\mathcal{W}^*|\meg (1-\gamma_0)|\mathcal{W}|$
and satisfying the following. For every $w\in\mathcal{W}^*$ there exists $\mathcal{N}_w^*\subseteq \{0,...,h-1\}$ 
with $|\mathcal{N}_w^*|\meg(1-\gamma_1-\varrho^3)h$ such that, setting
\begin{equation} \label{4e2new}
\Delta_w^*=\{s\in\mathcal{S}:f(s,w)\meg \alpha-\gamma_0-\gamma_1-\gamma_2\},
\end{equation}
we have $|\Delta_w^*\cap\mathcal{S}_n|\meg (1-\gamma_2)|\mathcal{S}_n|$ for every $n\in\mathcal{N}_w^*$.
\end{enumerate}
\end{lem}
\begin{proof}
We will consider four cases. The first three cases imply that alternative (i) holds true while the last one yields
alternative (ii). First we need to do some preparatory work. Precisely notice that, by (\ref{e64}), we have 
$\gamma_0\mik 4^{-4}$. Therefore, by property ($\mathcal{P}$2), we see that $\varrho\mik 4^{-4}$. Also, for every 
$w\in\mathcal{W}$ we set
\begin{equation} \label{e66}
\Delta_{w}=\{s\in\mathcal{S}:f(s,w)\meg\beta+\varrho^2/2\},
\end{equation}
\begin{equation} \label{e67}
I_{w}=\{ n<h: \ave_{s\in\mathcal{S}_n} f(s,w)\meg \beta+\varrho^2\}
\end{equation}
and
\begin{equation} \label{e68}
K_{w}=\{ n<h: |\Delta_{w}\cap \mathcal{S}_n|\meg \varrho^3 |\mathcal{S}_n|\}.
\end{equation}
We are now ready to consider cases.
\medskip

\noindent \textsc{Case 1:} \textit{there exists $w_0\in\mathcal{W}$ such that 
$\ave_{n<h}\ave_{s\in\mathcal{S}_n} f(s,w_0)\meg\beta+\varrho^2$}. In this case we set
\begin{equation} \label{4e3new}
\mathcal{N}_0=\{n<h:\ave_{s\in\mathcal{S}_n} f(s,w_0)\meg\beta+3\varrho^2/4\}.
\end{equation}
Recall that $\varrho\mik 4^{-1}$. Hence, by Fact \ref{markov-f1}, we see that $|\mathcal{N}_0|\meg \varrho^3h$.
Next we set $\Delta_0=\Delta_{w_0}$. Invoking Fact \ref{markov-f1} once again and using the definition of 
$\mathcal{N}_0$ we see that $|\Delta_0\cap\mathcal{S}_n|\meg \varrho^3|\mathcal{S}_n|$ for every $n\in\mathcal{N}_0$.
Therefore, this case implies part (i) of the lemma.
\medskip

\noindent \textsc{Case 2:} \textit{there exist  $w_0\in\mathcal{W}$ such that $|I_{w_0}|\meg \varrho^3 h$}.
By Fact \ref{markov-f1}, we see that $|\Delta_{w_0}\cap\mathcal{S}_n| \meg\varrho^3|\mathcal{S}_n|$ for every
$n\in I_{w_0}$. We set $\Delta_0=\Delta_{w_0}$ and $\mathcal{N}_0= I_{w_0}$ and we observe that with these 
choices the first part of the lemma is satisfied.
\medskip

\noindent \textsc{Case 3:} \textit{there exist $w_0\in\mathcal{W}$ such that $|K_{w_0}|\meg \varrho^3 h$}. 
In this case we set $\Delta_0=\Delta_{w_0}$ and $\mathcal{N}_0= K_{w_0}$. It is easily seen that with these
choices part (i) of the lemma is satisfied.
\medskip

\noindent \textsc{Case 4:} \textit{none of the above cases holds true}. Notice that, in this case, for every
$w\in\mathcal{W}$ we have
\begin{enumerate}
\item[(H1)] $\ave_{n<h}\ave_{s\in\mathcal{S}_n} f(s,w)<\beta+\varrho^2$,
\item[(H2)] $|I_{w}| < \varrho^3 h$ and
\item[(H3)] $|K_{w}| < \varrho^3 h$.
\end{enumerate}
We set $\alpha_0=\alpha-\gamma_0$ and
\begin{equation} \label{e612}
\mathcal{W}^*=\{ w\in\mathcal{W}: \ave_{n<h}\ave_{s\in\mathcal{S}_n} f(s,w)\meg \alpha_0\}.
\end{equation}
\begin{claim} \label{c63}
We have $|\mathcal{W}^*|\meg (1-\gamma_0)|\mathcal{W}|$.
\end{claim}
\begin{proof}[Proof of Claim \ref{c63}]
By assumption (c) of the lemma we have
\begin{equation} \ave_{w\in\mathcal{W}}\ave_{n<h}\ave_{s\in\mathcal{S}_n} f(s,w)\meg \alpha. \end{equation}
On the other hand, by property ($\mathcal{P}$1) and (H1), we see that
\begin{equation} \ave_{n<h}\ave_{s\in\mathcal{S}_n} f(s,w)<\alpha+\gamma_0^2 \end{equation}
for every $w\in\mathcal{W}$. Invoking Fact \ref{markov-f2} the result follows.
\end{proof}
Now we set $\alpha_1=\alpha_0-\gamma_1$. Observe that $\alpha_1=\alpha-\gamma_0-\gamma_1$. 
For every $w\in \mathcal{W}^*$ let
\begin{equation} \label{e617}
\widetilde{\mathcal{N}}_{w}=\{ n<h: \ave_{s\in\mathcal{S}_n} f(s,w)\meg \alpha_1\}.
\end{equation}
\begin{claim} \label{c64}
For every $w\in \mathcal{W}^*$ we have $|\widetilde{\mathcal{N}}_{w}|\meg (1-\gamma_1)h$.
\end{claim}
\begin{proof}[Proof of Claim \ref{c64}]
Let $w\in\mathcal{W}^*$ be arbitrary. Notice first that, by (H2), the definition of the set
$I_w$ in (\ref{e67}) and property ($\mathcal{P}$1), we have 
\begin{equation} |\{n<h:\ave_{s\in\mathcal{S}_n} f(s,w)\meg\alpha_0+\gamma_1^2\}|<\varrho^3 h. \end{equation}
On the other hand, since $w\in\mathcal{W}^*$ we have $\ave_{n<h}\ave_{s\in\mathcal{S}_n}f(s,w)\meg \alpha_0$. 
By Fact \ref{markov-f2} the result follows.
\end{proof}
Next we set $\alpha_2=\alpha_1-\gamma_2$ and we notice that $\alpha_2=\alpha-\gamma_0-\gamma_1-\gamma_2$. 
Finally, for every $w\in \mathcal{W}^*$ let
\begin{equation} \label{newe617}
\mathcal{N}^*_{w}=\widetilde{\mathcal{N}}_{w}\setminus K_{w}.
\end{equation}
Notice that, by Claim \ref{c64} and (H3), for every $w\in \mathcal{W}^*$ we have
\begin{equation} \label{4enewnew}
|\mathcal{N}^*_{w}|\meg (1-\gamma_1-\varrho^3)h.
\end{equation}
We will show that the set $\mathcal{W}^*$ defined in (\ref{e612}) and the family $\{\mathcal{N}^*_w:w\in\mathcal{W}^*\}$
satisfy part (ii) of the lemma. By Claim \ref{c63} and \eqref{4enewnew}, it is enough to show that for every
$w\in \mathcal{W}^*$ and every $n\in\mathcal{N}^*_{w}$ we have that 
$|\Delta^*_{w}\cap \mathcal{S}_n|\meg (1-\gamma_2)|\mathcal{S}_n|$. So, let $w\in \mathcal{W}^*$ and 
$n\in \mathcal{N}^*_{w}$ be arbitrary. Since $n\in\mathcal{N}^*_w\subseteq\widetilde{\mathcal{N}}_{w}$, 
by (\ref{e617}), we have $\ave_{s\in \mathcal{S}_n}f(s,w)\meg\alpha_1$. Moreover $n\notin K_{w}$, and so
\begin{eqnarray}
|\{s\in\mathcal{S}_n:f(s,w)\meg \alpha_1+\gamma_2^2\}| & \stackrel{(\mathcal{P}1)}{=} &
|\{s\in\mathcal{S}_n:f(s,w)\meg \beta+\varrho^2\}| \\
& \mik & |\{s\in\mathcal{S}_n:f(s,w)\meg \beta+\varrho^2/2\big\}| \nonumber \\
& \stackrel{(\ref{e66})}{=} & |\Delta_{w}\cap\mathcal{S}_n| \mik
\varrho^3|\mathcal{S}_n| \stackrel{(\mathcal{P}2)}{\mik} \gamma_2^3 h. \nonumber
\end{eqnarray}
By Fact \ref{markov-f2}, we conclude that $|\Delta^*_{w}\cap\mathcal{S}_n|\meg(1-\gamma_2)|\mathcal{S}_n|$. 
Thus, this case implies that part (ii) of the lemma is satisfied. The above cases are exhaustive, and so, 
the proof of Lemma \ref{l62} is completed.
\end{proof}
The following lemma is the second main result of this section.
\begin{lem} \label{newl63}
Let $\theta>0$. Assume that we are given
\begin{enumerate}
\item[(a)] two nonempty finite sets $\mathcal{S}$ and $\mathcal{W}$,
\item[(b)] an integer $h\meg 1$ and a partition $\{\mathcal{S}_0,...,\mathcal{S}_{h-1}\}$ of $\mathcal{S}$, and
\item[(c)] a function $g:\mathcal{S}\times\mathcal{W}\to [0, +\infty)$ such that
$\ave_{w\in\mathcal{W}}\ave_{n<h}\ave_{s\in \mathcal{S}_n} g(s,w)\mik \theta$.
\end{enumerate}
Then for every $0<\lambda< 1$ there exists $\mathcal{W}^*_\lambda\subseteq \mathcal{W}$ with 
$|\mathcal{W}^*_\lambda|\meg(1-\lambda)|\mathcal{W}|$ and satisfying the following. For every
$w\in \mathcal{W}_\lambda^*$ there exists $\mathcal{N}^*_{w,\lambda}\subseteq \{0,...,h-1\}$ with
$|\mathcal{N}^*_{w,\lambda}|\meg (1-\lambda)h$ such that, setting
\begin{equation} \label{4enewnewnew}
\Delta^*_{w,\lambda}=\{s\in\mathcal{S}: g(s,w)\mik \theta\lambda^{-3}\},
\end{equation}
we have $|\Delta^*_{w,\lambda}\cap \mathcal {S}_n|\meg (1-\lambda)|\mathcal{S}_n|$ for every 
$n\in\mathcal{N}^*_{w,\lambda}$.
\end{lem}
\begin{proof}
We fix $0<\lambda<1$. Let $\mathcal{W}^*_{\lambda}=\{w\in\mathcal{W}:\ave_{n<h}\ave_{s\in\mathcal{S}_n} g(s,w)\mik \theta\lambda^{-1}\}$. 
Also, for every $w\in\mathcal{W}_\lambda^*$ let $\mathcal{N}^*_{w,\lambda}=\{n<h:\ave_{s\in\mathcal{S}_n} g(s,w)\mik\theta\lambda^{-2}\}$.
Applying Fact \ref{markov-f3} successively three times, it is easy to see that the set $\mathcal{W}^*_{\lambda}$ 
and the family $\{\mathcal{N}^*_{w,\lambda}:w\in\mathcal{W}^*_\lambda\}$ satisfy the requirements of the lemma.
\end{proof}

\subsection{Consequences}

Lemmas \ref{l62} and \ref{newl63} will be used, later on, in a rather special form. We isolate, below,
the exact statement that we need.
\begin{cor} \label{comb1}
Let $0<\alpha\mik\beta\mik 1$ and $0<\varrho\mik 1$ and define $\gamma_0, \gamma_1$ and $\gamma_2$ as in
\eqref{e61}, \eqref{e62} and \eqref{e63} respectively. Also let 
$\theta,\lambda>0$ and $b, p,q\in\nn$ with $b,p,q\meg 1$. Assume that
\begin{equation} \label{ee64}
\gamma_0\mik \Big( \frac{\alpha}{12bp}\Big)^4 \text{ and } \lambda\mik \frac{\alpha}{12bq}.
\end{equation}
Assume, moreover, that we are given
\begin{enumerate}
\item[(a)] two nonempty finite sets $\mathcal{S}$ and $\mathcal{W}$,
\item[(b)] an integer $h\meg 1$ and a partition $\{\mathcal{S}_0,...,\mathcal{S}_{h-1}\}$ of $\mathcal{S}$,
\item[(c)] an integer $M\meg 1$ and a partition $\{\mathcal{W}_1,...,\mathcal{W}_M\}$ of $\mathcal{W}$
such that $|\mathcal{W}_k|=b$ for every $k\in [M]$,
\item[(d)] a subset $\mathcal{A}$ of $[M]$ with $|\mathcal{A}|\meg (\alpha/10)M$,
\item[(e)] for every $j\in [p]$ a function $f_j:\mathcal{S}\times\mathcal{W}\to [0,1]$ such that
\begin{equation} \ave_{w\in\mathcal{W}}\ave_{n<h}\ave_{s\in \mathcal{S}_n} f_j(s,w)\meg \alpha, \end{equation}
\item[(f)] for every $r\in [q]$ a function $g_r:\mathcal{S}\times\mathcal{W}\to [0, +\infty)$ such that
\begin{equation} \ave_{w\in\mathcal{W}}\ave_{n<h}\ave_{s\in \mathcal{S}_n}g_r(s,w)\mik\theta. \end{equation}
\end{enumerate}
Then, either
\begin{enumerate}
\item[(i)] there exist  $j_0\in [p]$, $w_0\in\mathcal{W}$ and $\mathcal{N}_0\subseteq \{0,...,h-1\}$
with $|\mathcal{N}_0|\meg \varrho^3 h$ such that, setting
\begin{equation} \label{4ecor1}
\Delta_0=\{s\in\mathcal{S}: f_{j_0}(s,w_0)\meg\beta+\varrho^2/2\},
\end{equation}
we have $|\Delta_0\cap\mathcal{S}_n|\meg \varrho^3|\mathcal{S}_n|$ for every $n\in\mathcal{N}_0$, or
\item[(ii)] there exist $k_0\in \mathcal{A}$ and $\mathcal{N}^{*}\subseteq \{0,...,h-1\}$ with
$|\mathcal{N}^{*}|\meg \varrho^3 h$ such that, setting $\alpha'=\alpha-\gamma_0-\gamma_1-\gamma_2$ and
\begin{equation} \label{4ecor2}
\Delta^{*}=\bigcap_{w\in\mathcal{W}_{k_0}} \bigcap_{j\in [p]}\bigcap_{r\in [q]} \{s\in\mathcal{S}: 
f_{j}(s,w)\meg\alpha' \text{ and } g_r(s,w)\mik \theta \lambda^{-3}\},
\end{equation}
we have $ |\Delta^{*}\cap \mathcal  {S}_n|\meg \varrho^3 h$ for every $n\in\mathcal{N}^{*}$.
\end{enumerate}
\end{cor}
\begin{proof}
Assume that alternative (i) is not satisfied. We will show that part (ii) of the corollary holds true. 
Notice first that, by Lemma \ref{l62}, for every $j\in [p]$ there exists $\mathcal{W}_j^*\subseteq\mathcal{W}$ 
with $|\mathcal{W}_j^*|\meg (1-\gamma_0)|\mathcal{W}|$ and satisfying the following. For every $w\in\mathcal{W}_j^*$ 
there exists $\mathcal{N}_{w, j}^*\subseteq \{0,...,h-1\}$ with $|\mathcal{N}_{w,j}^*|\meg(1-\gamma_1-\varrho^3)h$
such that, setting $\Delta_{w,j}^*=\{s\in\mathcal{S}:f_{j}(s,w)\meg \alpha-\gamma_0-\gamma_1-\gamma_2 \big\}$,
we have $|\Delta_{w,j}^*\cap \mathcal{S}_n|\meg (1-\gamma_2)|\mathcal{S}_n|$ for every $w\in\mathcal{W}_j^*$ 
and every $n\in\mathcal{N}_{w,j}^*$.

Next observe that, by Lemma \ref{newl63}, for every $r\in [q]$ there exists $\mathcal{W}^*_{\lambda,r}\subseteq \mathcal{W}$ 
with  $|\mathcal{W}^*_{\lambda,r}|\meg(1-\lambda)|\mathcal{W}|$ and satisfying the following. For every 
$w\in \mathcal{W}^*_{\lambda,r}$ there exists $\mathcal{N}^*_{w,\lambda,r}\subseteq \{0,...,h-1\}$ with
$|\mathcal{N}^*_{w,\lambda,r}|\meg (1-\lambda)h$ such that, setting 
\begin{equation} \Delta^*_{w,\lambda,r}=\{s\in\mathcal{S}:g_r(s,w)\mik \theta\lambda^{-3}\}, \end{equation}
we have $|\Delta^*_{w,\lambda,r}\cap\mathcal {S}_n| \meg (1-\lambda)|\mathcal{S}_n|$ for every
$n\in\mathcal{N}^*_{w,\lambda}$. We set
\begin{equation} \label{eqwwss}
\mathcal{W}^{*}=\bigcap_{j\in [p]}\mathcal{W}_j^*\cap\bigcap_{r\in [q]}\mathcal{W}^*_{\lambda,r}
\end{equation}
and we notice that
\begin{equation} \label{eqstar}
|\mathcal{W}^{*}|\meg (1-p\gamma_0-q\lambda)|\mathcal{W}|.
\end{equation}
\begin{claim} \label{c64comb}
There exists $k_0\in \mathcal{A}$ such that $\mathcal{W}_{k_0}\subseteq\mathcal{W}^{*}$.
\end{claim}
\begin{proof}[Proof of Claim \ref{c64comb}]
Let $\mathcal{B}=\{k\in [M]:\mathcal{W}_k\nsubseteq \mathcal{W}^*\}$. Clearly it is enough to show that
$|\mathcal{B}|<|\mathcal{A}|$. To this end observe that, by \eqref{ee64}, we have
\begin{equation} \label{pq}
bp\gamma_0+bq\lambda\mik\Big(\frac{\alpha}{12}\Big)^4+\frac{\alpha}{12}<\frac{\alpha}{10}.
\end{equation}
Since $\{\mathcal{W}_1,...,\mathcal{W}_M\}$ is a partition of $\mathcal{W}$ with $|\mathcal{W}_k|=b$ 
or every $k\in [M]$, we get
\begin{equation} \label{4ecor3}
|\mathcal{B}|\mik |\mathcal{W}\setminus \mathcal{W}^*| \stackrel{(\ref{eqstar})}{\mik} (p\gamma_0+q\lambda)
|\mathcal{W}| \mik (p\gamma_0+q\lambda)bM\stackrel{(\ref{pq})}{<}(\alpha/10)M\mik|\mathcal{A}|
\end{equation}
and the claim is proved.
\end{proof}
Now let
\begin{equation} \label{enstar}
\mathcal{N}^{*}=\bigcap_{w\in\mathcal{W}_{k_0}}\Big(\bigcap_{j\in [p]}\mathcal{N}_{w,j}^*\cap
\bigcap_{r\in [q]}\mathcal{N}_{w,\lambda,r}^*\Big).
\end{equation}
We will show that the integer $k_0$ obtained by Claim \ref{c64comb} and the set $\mathcal{N}^*$
satisfy the requirements imposed in the second part of the corollary. First we will show that
$|\mathcal{N}^{*}|\meg \varrho^3 h$. Since $\varrho\mik\gamma_0\mik 1/12^4$, it is enough to prove
that $|\mathcal{N}^*|\meg (3/4)h$. To this end notice that, by Claim \ref{c64comb} and part (c) of
our assumptions, we have $\mathcal{W}_{k_0}\subseteq \mathcal{W}^{*}$ and $|\mathcal{W}_{k_0}|=b$.
Now recall that $|\mathcal{N}_{w,j}^*|\meg(1-\gamma_1-\varrho^3)h$ for every $j\in [p]$ and every
$w\in\mathcal{W}^*_j$ while $|\mathcal{N}^*_{w,\lambda,r}|\meg (1-\lambda)h$ for every $r\in [q]$
and every $w\in\mathcal{W}^*_{\lambda,r}$. Taking into account these remarks and invoking (\ref{eqwwss})
and (\ref{enstar}), we see that
\begin{equation} \label{er}
|\mathcal{N}^{*}|\meg \big(1-b(p\gamma_1+p\varrho^3+q\lambda)\big)h.
\end{equation}
Using (\ref{ee64}) and properties ($\mathcal{P}$2) and ($\mathcal{P}$3) isolated after (\ref{e63}), we also have that
\begin{equation} \label{4ecornew}
b(p\gamma_1+p\varrho^3+q\lambda)\mik b\big( p(2\gamma_0)^{1/2}+p\gamma_0^3+q\lambda\big)<3\frac{1}{12}=\frac{1}{4}.
\end{equation}
Therefore, $|\mathcal{N}^{*}|\meg (3/4)h$ as claimed. Next we work to show that for every $n\in\mathcal{N}^*$ we 
have that $|\Delta^{*}\cap \mathcal {S}_n|\meg \varrho^3 |\mathcal{S}_n|$. Observe that it is enough to prove that
\begin{equation} |\Delta^{*}\cap \mathcal {S}_n|\meg (3/4)|\mathcal{S}_n| \end{equation}
for every $n\in\mathcal{N}^*$. So let $n\in\mathcal{N}^*$ be arbitrary. Notice that
\begin{equation} \label{4ecornewnew}
\Delta^*=\bigcap_{w\in\mathcal{W}_{k_0}}\Big(\bigcap_{j\in [p]}\Delta^*_{w,j}\cap\bigcap_{r\in [q]}\Delta_{w,\lambda,r}^*\Big).
\end{equation}
Arguing as above and using the estimates for the size of the sets $\Delta^*_{w,j}\cap\mathcal{S}_n$ and 
$\Delta_{w,\lambda,r}^*\cap\mathcal{S}_n$ we see that
\begin{equation} \label{eqh12}
|\Delta^{**}\cap \mathcal{S}_n|\meg \big(1-b(p\gamma_2+q\lambda)\big)|\mathcal{S}_n|.
\end{equation}
Invoking (\ref{ee64}) and property ($\mathcal{P}$3) once again, we have
\begin{equation} \label{4ecorreallynew}
b(p\gamma_2+q\lambda)\mik b(p2\gamma_0^{1/4}+q\lambda)\mik \frac{2\alpha}{12}+\frac{\alpha}{12}\mik \frac{1}{4}.
\end{equation}
Hence, $|\Delta^{*}\cap \mathcal{S}_n|\meg (3/4)|\mathcal{S}_n|$. The proof of Corollary \ref{comb1} is thus completed.
\end{proof}


\section{Outline of the proof of Theorem \ref{it5}}

\numberwithin{equation}{section}

In this section we shall give a detailed outline of the proof of Theorem \ref{it5}. Very briefly, 
and oversimplifying dramatically, the proof is based on a density increment strategy and is similar in spirit
to the proof of \cite[Lemma 27]{DKT2}. There are, however, significant differences and, therefore, a novelty
of the approach. The most important one is that for a fixed ``dimension" $d$, we need to apply Theorem 
\ref{uniform dhl} for a ``dimension" $d'$ which is much bigger compared to $d$. We shall comment further
on this feature of the proof below.

To proceed with our discussion we need, first, to introduce the notions of ``strong denseness" and 
``strong negligibility". These concepts are our main conceptual tools and are critical for the proof
of Theorem \ref{it5}. We start with the following definition (see, also, \cite[Definition 14]{DKT2}).
\begin{defn} \label{nd70}
Let $D:\otimes\bfct\to \mathcal{P}(W)$ be a level selection. Also let $\bfcs$ be a vector strong subtree 
of $\bfct$, $w\in W$ and $0<\alpha\mik 1$.
\begin{enumerate}
\item[(1)] We say that $D$ is \emph{$(w,\bfcs,\alpha)$-dense} if for every $\bfs\in\otimes\bfcs$ and
every $r\in D(\bfs)$ we have $\ell_W(w)\mik \ell_W(r)$ and, moreover,  $\dens(D(\bfs) \ | \ w)\meg \alpha$
for every $\bfs\in\otimes\bfcs$.
\item[(2)] We say that $D$ is \emph{$(w,\bfcs,\alpha)$-strongly dense} if $D$ is $(w',\bfcs,\alpha)$-dense
for every $w'\in\immsuc_W(w)$.
\end{enumerate}
\end{defn}
Next we introduce the notions of ``negligibility" and ``strong negligibility".
\begin{defn} \label{nd61}
Let $D:\otimes\bfct\to \mathcal{P}(W)$ be a level selection. Also let $F\subseteq \otimes\bfct(m)$ for
some $m<h(\bfct)$, $w\in W$ and $0<\theta\mik 1$.
\begin{enumerate}
\item[(1)] We say that the pair $(F, w)$ is $\theta$-\emph{negligible with respect to} $D$ if for every
$\bft\in F$ and every $r\in D(\bft)$ we have $\ell_W(w)\mik \ell_W(r)$ and, moreover,
$\dens\big(\bigcap_{\bft\in F}D(\bft) \ | \ w\big)<\theta$.
\item[(2)] We say that the pair $(F, w)$ is \emph{strongly}  $\theta$-\emph{negligible with respect to}
$D$ if $(F, w')$ is $\theta$-negligible for every $w'\in\immsuc_W(w)$.
\end{enumerate}
\end{defn}
Notice that ``negligibility" is, essentially, the converse of the notion of ``strong correlation" introduced
in Definition \ref{id4}. Indeed, let $D:\otimes\bfct\to\mathcal{P}(W)$ be a level selection, 
$\bfcf\in\strong_2(\bfct)$ and $w\in D\big(\bfcf(0)\big)$. Also let $0<\theta\mik 1$. Then observe that either
the pair $(\bfcf,w)$ is strongly $\theta$-correlated with respect to $D$, or there exists $p\in\{0,...,b_W-1\}$ 
such that the pair $(\otimes\bfcf(1),w^{\con_W}\!p)$ is $\theta$-negligible.

After this preliminary discussion we are ready to comment on the proof. So let $\bfct$ be a finite vector homogeneous
tree with $b_{\bfct}=(b_1,...,b_d)$, $W$ a homogeneous tree with $b_W=b_{d+1}$ and  $D:\otimes\bfct\to \mathcal{P}(W)$
a level selection of density $\ee$ and of sufficiently large height. Recall that we need to find a pair $(\bfcf,w)$
which is strongly $\theta$-correlated with respect to $D$,  where $\theta$ is an appropriately chosen positive
constant that depends only on $b_1,...,b_d, b_{d+1}$ and $\ee$. 

The first observation we make -- an observation which is quite standard in proofs of this sort -- is that we
can assume that we have ``lack of density increment".  This means that we cannot increase, significantly, the
density of the level selection $D$ by restricting its values to a subtree of the form $\suc_W(w)$ for some
$w\in W$. This assumption is easily seen to be equivalent to a strong ``concentration hypothesis" for a 
probability measure on $\otimes\bfct$ introduced by H. Furstenberg and B. Weiss in \cite{FW}. We will not
comment further on this part of the proof. Instead we refer the reader to \cite[\S 4]{DKT2} for a detailed
exposition.

Next, with the ``concentration hypothesis" at hand, we devise a ``greedy" algorithm in order to find the
desired pair $(\bfcf,w)$. This algorithm will terminate after at most $K=\udhl\big(b_1,...,b_d|2,\ee/(4b_{d+1})\big)$
iterations. Of course, it is crucial that the number of iterations is a priori controlled. 

After the $n$-th iteration we will have as an input a vector strong subtree $\bfcz_n$ of $\bfct$ of sufficiently
large height and a node $\tilde{w}_n\in W$ such that the level selection $D$ is 
$(\tilde{w}_n,\suc_{\bfcz_n}(\bfz),\ee_n)$-dense for every $\bfz\in \otimes\bfcz_n(n)$, where $\ee_n$ is roughly
equal to $\ee$. We will also be given a subset $\Gamma_n$ of $\otimes\bfcz_n\upharpoonright (n-1)$ and a small
constant $\theta_n$ such that the pair $(\otimes\bfcf(1),\tilde{w}_n)$ is $\theta_n$-negligible for every 
$\bfcf\in \strong_2(\bfcz_n)$ with $\bfcf(0)\in\Gamma_n$ and $\otimes\bfcf(1) \subseteq \otimes\bfcz_n(m)$
for some $m\meg n$. 

First we show that we can select a node $w\in \suc_W(\tilde{w}_n)$ and a vector strong subtree $\bfcs$ of
$\bfcz_n$ of sufficiently large height, with $\bfcs\upharpoonright n=\bfcz_n\upharpoonright n$ and such that
the pair $(\otimes\bfcf(1),w)$ is strongly $\theta_{n+1}$-negligible for every $\bfcf\in \strong_2(\bfcs)$ with
$\bfcf(0)\in\Gamma_n$ and $\otimes\bfcf(1) \subseteq \otimes\bfcs(m)$ for some $m\meg n+1$. Here $\theta_{n+1}$
is a numerical parameter which is effectively controlled by $\theta_n$. It is precisely in this step that we need
to invoke Theorem \ref{uniform dhl} for a ``dimension" $d'$ which is bigger than $d$.  Moreover, using the
``concentration hypothesis" mentioned above, we can ensure that the level selection $D$ is
$(w,\suc_{\bfcs}(\bfs),\ee')$-strongly dense for every $\bfs\in \otimes\bfcs(n+1)$, where $\ee'$ is another
numerical parameter that can be arranged to be as close to $\ee$ as we wish. Finally, the node $w$ is chosen
so that the set $B:=\{\bfs\in\otimes\bfcs(n): w\in D(\bfs)\}$ has density at least $\ee/4$ and, on
the other hand, $w\notin D_{\bfcf}$ for every $\bfcf\in \strong_2(\bfcs,n)$ with $\bfcf(0)\in \Gamma_n$.
The details for this selection are presented in Lemma \ref{c67}.

The second part of the argument is based on an application of Milliken's Theorem and is the content of Lemma
\ref{c72} in the main text. Specifically, if we cannot find a pair $(\bfcf,w)$ which is strongly 
$\theta_{n+1}$-correlated with respect to $D$, then we may select $p\in\{0,...,b_{d+1}-1\}$, a vector strong
subtree $\bfcz$ of $\bfcs$ of sufficiently large height and with $\bfcz\upharpoonright n=\bfcs\upharpoonright n$,
as well as, a subset $B'$ of $B$ of density at least $\ee/(4 b_{d+1})$ such that the pair $(\otimes\bfcf(1),w^{\con_W}\!p)$
is $\theta_{n+1}$-negligible for every $\bfcf\in\strong_2(\bfcz)$ with $\bfcf(0)\in B'$.  We set
``$\tilde{w}_{n+1}=w^{\con_W}\!p$", ``$\bfcz_{n+1}=\bfcz$" and ``$\Gamma_{n+1}=\Gamma_n\cup B'$" and we proceed
to the next iteration.

If after $K$ many iterations the desired pair $(\bfcf,w)$ has not been found, then using the sets 
$\Gamma_1,...,\Gamma_K$ construct above we can easily derive a contradiction. This implies, of course, that the
algorithm will eventually terminate, completing thus the proof of Theorem \ref{it5}.


\section{Preliminary tools}

\numberwithin{equation}{section}

As we have already mentioned, our goal in this section is to develop the main tools needed for the proof
of Theorem \ref{it5}. The first one is the following lemma.
\begin{lem} \label{c67}
Let $d\in\nn$ with $d\meg 1$ and $b_1,...,b_d, b_{d+1}\in\nn$ with $b_i\meg 2$ for all $i\in [d+1]$. 
Let $0<\alpha\mik\beta\mik 1$, $0<\varrho\mik 1$ and define $\gamma_0$, $\gamma_1$ and $\gamma_2$ as in \eqref{e61},
\eqref{e62} and \eqref{e63} respectively. Also let $m\in\nn$ and $\theta,\lambda>0$. Set $q_m=q(b_1,...,b_d,m)$, where
$q(b_1,...,b_d,m)$ is defined in \eqref{e214}, and assume that
\begin{equation} \label{ee626}
\gamma_0\mik \Big( \frac{\alpha}{12b_{d+1}(\prod_{i=1}^d b_i)^{m+1}}\Big)^4, \ \lambda\mik \frac{\alpha}{12b_{d+1}q_m} 
\text{ and } \theta\mik\frac{2\alpha}{5q_m}.
\end{equation}
Assume, moreover, that we are given
\begin{enumerate}
\item[(a)] a homogeneous tree $W$ with $b_W=b_{d+1}$,
\item[(b)] a finite vector homogeneous tree $\bfcz=(Z_1,...,Z_d)$ with $b_{\bfcz}=(b_1,...,b_d)$,
\item[(c)] a level selection $D:\otimes\bfcz\to\mathcal{P}(W)$ and $\ell\in\nn$ such that $D(\bfz)\subseteq W(\ell)$
for every $\bfz\in\otimes\bfcz(m)$,
\item[(d)] a node $\tilde{w}\in W$ with $\ell_W(\tilde{w})\mik\ell$ and such that $D$ is 
$(\tilde{w},\suc_{\bfcz}(\bfz),\alpha)$-dense for every $\bfz\in\otimes\bfcz(m)$,
\item[(e)] if $m\meg 1$, a subset $\Gamma$ of $\otimes\bfcz\upharpoonright(m-1)$ such that
$(\otimes\bfcg(1),\tilde{w})$ is $\theta$-negligible with respect to $D$ for every
$\bfcg\in\bigcup_{n=m}^{h(\bfcz)-1}\strong_2(\bfcz,n)$ with $\bfcg(0)\in\Gamma$.
\end{enumerate}
Finally, let $N\in\nn$ with $N\meg 1$ and suppose that
\begin{equation} \label{ne627}
h(\bfcz)\meg (m+1)+ \varrho^{-3} \udhl\big(\underbrace{b_1,...,b_1}_{b_1^{m+1}-\mathrm{times}}, ..., 
\underbrace{b_d,...,b_d}_{b_d^{m+1}-\mathrm{times}}|N,\varrho^3\big).
\end{equation}
Then, either
\begin{enumerate}
\item[(I)] there exist $w'\in W(\ell+1)\cap\suc_W(\tilde{w})$ and a vector strong subtree $\bfcz'$ of $\bfcz$
with $h(\bfcz')=N$ such that $D$ is $(w',\bfcz',\beta+\varrho^2/2)$-dense, or
\item[(II)] there exist $w''\in W(\ell)\cap\suc_W(\tilde{w})$, a vector strong subtree $\bfcz''$ of $\bfcz$
with $\bfcz''\upharpoonright m=\bfcz\upharpoonright m$ and $h(\bfcz'')=(m+1)+N$, as well as,
$B\subseteq \otimes \bfcz(m)$ with $|B|\meg (\alpha/2)|\!\otimes \bfcz(m)|$ satisfying the following.
\begin{enumerate}
\item[(II1)]  We have $w''\in \bigcap_{\bfz\in B}D(\bfz)$. On the other hand, 
$w''\notin \bigcap_{\bfz\in\otimes\bfcg(1)}D(\bfz)$ for every $\bfcg \in\strong_2(\bfcz,m)$ with $\bfcg(0)\in \Gamma$.
\item[(II2)]  The level selection $D$ is $(w'',\suc_{\bfcz''}(\bfz),\alpha-\gamma_0-\gamma_1-\gamma_2)$-strongly
dense for every $\bfz\in\otimes\bfcz''(m+1)$.
\item[(II3)] If $\bfcg\in\strong_2(\bfcz'',n)$ for some $n\in\{m+1,..., h(\bfcz'')-1\}$ with $\bfcg(0)\in\Gamma$, 
then the pair $(\otimes\bfcg(1),w'')$ is strongly $(\theta\lambda^{-3})$-negligible with respect to the level selection $D$.
\end{enumerate}
\end{enumerate}
\end{lem}
\begin{proof}
We will give the proof under the assumption that $m\meg 1$. If $m=0$, then the proof is similar and, in fact, 
simpler since Step 4 below is not needed. We will use Corollary \ref{comb1}. To this end we need, of course, to define
all necessary data. First we set $\mathcal{W}=W(\ell+1)\cap\suc_W(\tilde{w})$. Also we will define
\begin{enumerate}
\item[(1)] a finite vector homogeneous tree $\bfcs=(S_1,...,S_n)$ with $n=\sum_{i=1}^{d}b_i^{m+1}$,
\item[(2)] a subset $A$ of $W(\ell)\cap\suc_W(\tilde{w})$,
\item[(3)] for every $\bfz\in\otimes\bfcz(m+1)$ a function $f_\bfz:\otimes\bfcs\times \mathcal{W}\to [0,1]$ and
\item[(4)] for every $\bfcf\in \strong_2(\bfcz, m+1)$ a function $g_\bfcf:\otimes\bfcs\times \mathcal{W}\to [0,1]$.
\end{enumerate}

\subsubsection*{Step 1: defining the finite vector homogeneous tree $\bfcs$}

For every $i\in [d]$ let $p_i=b_i^{m+1}$ and notice that the cardinality of $Z_i(m+1)$ is $p_i$. 
Let $\{z^i_1<_{\text{lex}}... <_{\text{lex}} z^i_{p_i}\}$ be the lexicographical increasing enumeration of $Z_i(m+1)$.
For every $i\in [d]$ and every $j\in [p_i]$ let $S^i_j=\suc_{Z_i}(z^i_j)$ and set $\bfcs=(S^i_j)_{i=1, j=1}^{d \ \ \ p_i}$. 
It is clear that $\bfcs$ is a finite vector homogeneous tree with $h(\bfcs)=h(\bfcz)-(m+1)$. Moreover,
\begin{equation} \label{neqhei}
b_\bfcs=\big(\underbrace{b_1,...,b_1}_{b_1^{m+1}-\mathrm{times}}, ..., \underbrace{b_d,...,b_d}_{b_d^{m+1}-\mathrm{times}}\big).
\end{equation}
The vector homogeneous trees $\bfcz$ and $\bfcs$ are naturally associated. Precisely, let $\bfz=(z_1,...,z_d)\in\otimes\bfcz(m+1)$ 
and $i\in [d]$ be arbitrary. Our assumptions permits us to define the ``projection" $\pi^i_{\bfz}:\otimes\bfcs\to \suc_{Z_i}(z_i)$. 
Formally it is defined as follows: if $\bfs=(s^i_j)_{i=1, j=1}^{d \ \ \ p_i}\in\otimes\bfcs$, then $\pi^i_{\bfz}(\bfs)$
is the unique node $s^i_j$ such that $s^i_j\in\suc_{Z_i}(z_i)$. Notice that $\pi^i_{\bfz}$ is onto. We will also need
the ``full-projection" $\Pi_{\bfz}:\otimes\bfcs\to\otimes\suc_{\bfcz}(\bfz)$ defined by  
$\Pi_{\bfz}(\bfs)=\big( \pi^1_{\bfz}(\bfs),..., \pi^d_{\bfz}(\bfs)\big)$. Clearly $\Pi_{\bfz}$ is onto.

\subsubsection*{Step 2: defining the set $A$}

Let $C$ be the subset of $W(\ell)\cap \suc_W(\tilde{w})$ defined by
\begin{equation} \label{ee920}
w\in C \Leftrightarrow |\{\bfz\in\otimes\bfcz(m): w\in D(\bfz)\}|\meg (\alpha/2)|\!\otimes\bfcz(m)|.
\end{equation}
Using condition (d) of the statement of the lemma, it is easy to verify that
\begin{equation} |C|\meg (\alpha/2)|W(\ell)\cap \suc_W(\tilde{w})|. \end{equation}
Also let $\mathcal{G}=\{\bfcg\in\strong_2(\bfcz,m): \bfcg(0)\in\Gamma\}$. Invoking hypothesis (e) and Fact \ref{f21}, we get that
\begin{eqnarray} \label{ee9219}
\dens\Big(\bigcup_{\bfcg\in \mathcal{G}} \bigcap_{\bfz\in\otimes\bfcg(1)} D(\bfz) \ | \ \tilde{w}\Big) & \mik &
|\strong_2(\bfcz,m)| \theta \\
& < & |\strong_2(\bfcz,m+1)| \theta =q_m\theta \stackrel{(\ref{ee626})}{\mik}2\alpha/5. \nonumber
\end{eqnarray}
We set
\begin{equation} \label{ee92199}
A= C \setminus \Big( \bigcup_{\bfcg\in\mathcal{G}} \bigcap_{\bfz\in\otimes\bfcg(1)} D(\bfz)\Big).
\end{equation}
Combining the previous estimates, we see that $|A|\meg(\alpha/10) |W(\ell)\cap\suc_W(\tilde{w})|$. 

\subsubsection*{Step 3: defining the family $\{f_\bfz :\bfz\in\otimes\bfcz(m+1)\}$}

For every $\bfz\in\otimes\bfcz(m+1)$ define $f_\bfz: \otimes\bfcs\times \mathcal{W}\to [0,1]$ by
\begin{equation}\label{nw77}
f_\bfz(\bfs, w)=\dens\big( D\big(\Pi_{\bfz}(\bfs)\big) \ | \ w\big).
\end{equation}
Notice for every $\bfz\in\otimes\bfcz(m+1)$ we have
\begin{equation} \label{neqalph}
\mathbb{E}_{w\in \mathcal{W}}\mathbb{E}_{n<h(\bfcs)}\mathbb{E}_{\bfs\in\otimes\bfcs(n)}f_\bfz(\bfs,w)\meg\alpha.
\end{equation}
Indeed, to verify (\ref{neqalph}) it is enough to show that for every $\bfz\in\otimes\bfcz(m+1)$ and every $\bfs\in \otimes\bfcs$ we have
$\mathbb{E}_{w\in \mathcal{W}}f_\bfz(\bfs,w)\meg\alpha$. So let $\bfz\in\otimes\bfcz(m+1)$ and $\bfs\in \otimes\bfcs$ be arbitrary. Also let
$L(D)=\{\ell_n: 0\mik n<h(\bfcz)\}$ be the level set of $D$. By conditions (c) and (d), we have $\ell_W(\tilde{w})\mik\ell=\ell_m$. Moreover, 
$\Pi_{\bfz}(\bfs)\in\otimes\suc_{\bfcz}(\bfz)$ and so $D\big(\Pi_{\bfz}(\bfs)\big)\subseteq W(l_k)$ for some $k\in \{m+1,...,h(\bfcz)-1\}$. 
Taking into account these remarks and using the fact that the tree $W$ is homogeneous and condition (d), we conclude that
\begin{eqnarray}
\mathbb{E}_{w\in \mathcal{W}}f_\bfz(\bfs,w) & = & \mathbb{E}_{w\in W(\ell+1)\cap\suc_W(\tilde{w})} \dens\big(D\big(\Pi_{\bfz}(\bfs)\big) \ | \ w\big) \\
& = & \dens\big(D\big(\Pi_{\bfz}(\bfs)\big) \ | \ \tilde{w}\big)\meg \alpha. \nonumber
\end{eqnarray}

\subsubsection*{Step 4: defining the family $\{g_\bfcf: \bfcf\in\strong_2(\bfcz,m+1)\}$}

In this step we need, first, to do some preparatory work. Specifically, let $\bfcf=(F_1,...,F_d)\in\strong_2(\bfcz,m+1)$
and $\bfs\in\otimes \bfcs$ be arbitrary. For every $i\in [d]$ let 
$G^{\bfcf}_{\bfs,i}=\{F_i(0)\}\cup \{\pi^i_{\bfz}(\bfs):\bfz\in \otimes \bfcf(1)\}$ and define 
$\mathbf{G}^{\bfcf}_{\bfs}=\big(G^{\bfcf}_{\bfs,1},...,G^{\bfcf}_{\bfs,d}\big)$. It is easy to see that
$\mathbf{G}^{\bfcf}_{\bfs}$ is a vector strong subtree of $\bfcz$ of height $2$ and with $\bfcg^{\bfcf}_{\bfs}(0)=\bfcf(0)$.
In fact, notice that $\mathbf{G}^{\bfcf}_{\bfs}\in\strong_2(\bfcz,n)$ for some  $n\in\{m+1,...,h(\bfcz)-1\}$ depending
only on the length of $\bfs$.

Now, for every $\bfcf\in\strong_2(\bfcz,m+1)$ we define $g_\bfcf: \otimes\bfcs\times \mathcal{W} \to [0,1]$ by
\begin{equation}
g_\bfcf(\bfs, w)=\dens\Big(\bigcap_{\bfz\in\otimes \mathbf{G}^{\bfcf}_{\bfs}(1)} D(\bfz) \ | \ w\Big).
\end{equation}
We isolate, for future use, the following fact: \textit{for every $\bfcf\in\strong_2(\bfcz,m+1)$ with $\bfcf(0)\in\Gamma$ we have}
\begin{equation}\label{neqthe}
\mathbb{E}_{w\in \mathcal{W}} \mathbb{E}_{n<h(\bfcs)} \mathbb{E}_{\bfs\in\otimes\bfcs(n)} g_\bfcf(\bfs,w)<\theta.
\end{equation}
This can be easily checked using condition (e) and arguing as in Step 3.

\subsubsection*{Applying Corollary \ref{comb1}}

Recall that we have already set $\mathcal{W}=W(\ell+1)\cap\suc_W(\tilde{w})$. Let $\mathcal{S}=\otimes \bfcs$, 
$h=h(\bfcs)$ and $\mathcal{S}_n=\otimes\bfcs(n)$ for every $n\in\{0,...,h(\bfcs)-1\}$. Also let $M=|W(\ell)\cap\suc_W(\tilde{w})|$
and enumerate the set $W(\ell)\cap\suc_W(\tilde{w})$ as $\{w_k:k\in [M]\}$. Moreover, set $\mathcal{A}=\{k\in [M]:w_k\in A\}$
and $\mathcal{W}_k=\immsuc_W(w_{k})$ for every $k\in [M]$. It is clear that $\{\mathcal{S}_0,...,\mathcal{S}_{h-1}\}$ and 
$\{\mathcal{W}_1,...,\mathcal{W}_M\}$ are partitions of $\mathcal{S}$ and $\mathcal{W}$ respectively. Since the branching
number of the tree $W$ is $b_{d+1}$, we see that $|\mathcal{W}_k|=b_{d+1}$ for every $k\in [M]$. Also, by the estimate on
the size of the set $A$ obtained in Step 2, we have that $|\mathcal{A}|\meg (\alpha/10)M$.

Next set $p=|\!\otimes\bfcz(m+1)|$. Also let $\mathcal{F}=\{\bfcf\in\strong_2(\bfcz,m+1): \bfcf(0)\in\Gamma\}$ and denote
by $q$ the cardinality of the set $\mathcal{F}$.  Clearly $p=\prod_{i=1}^db_i^{m+1}$ and, by Fact \ref{f21}, $q\mik q_{m}$.
Hence, by (\ref{ee626}), we see that
\begin{equation}
\gamma_0\mik \Big( \frac{\alpha}{12b_{d+1}p} \Big)^4 \text{ and } \lambda\mik \frac{\alpha}{12b_{d+1}q}.
\end{equation}
By the above discussion and taking into account (\ref{neqalph}) and (\ref{neqthe}), we conclude that Corollary \ref{comb1}
can be applied for the maps $\{f_{\bfz}:\bfz\in\otimes\bfcz(m+1)\}$ and $\{g_{\bfcf}:\bfcf\in\mathcal{F}\}$ and the data we
described above.

Assume, first, that part (i) of Corollary \ref{comb1} is satisfied. Therefore, there exist $\bfz_0\in\otimes\bfcz(m+1)$, 
$w_0\in W(\ell+1)\cap\suc_W(\tilde{w})$ and $\mathcal{N}_0\subseteq \{0,...,h(\bfcs)-1\}$ with $|\mathcal{N}_0|\meg \varrho^3 h(\bfcs)$ such that
\begin{equation} \label{b76}
|\Delta_0\cap\otimes\bfcs(n)|\meg \varrho^3|\otimes\bfcs(n)|
\end{equation}
for every $n\in\mathcal{N}_0$, where $\Delta_0=\big\{\bfs\in\otimes\bfcs:f_{\bfz_0}(\bfs,w_0)\meg\beta+\varrho^2/2\big\}$. Notice that
\begin{eqnarray} \label{estimate-n_0}
\mathcal{N}_0 & \meg & \varrho^3h(\bfcs)= \varrho^3 \big(h(\bfcz)-(m+1)\big) \\
& \stackrel{(\ref{ne627})}{\meg} & \udhl\big(\underbrace{b_1,...,b_1}_{b_1^{m+1}-\mathrm{times}}, ...,
\underbrace{b_d,...,b_d}_{b_d^{m+1}-\mathrm{times}}|N,\varrho^3\big).\nonumber
\end{eqnarray}
By (\ref{neqhei}), (\ref{b76}), (\ref{estimate-n_0}) and Theorem \ref{uniform dhl}, there exists a vector strong subtree
$\bfcs'$ of $\bfcs$ with $h(\bfcs')=N$ and such that $\otimes\bfcs'\subseteq \Delta_0$. We set 
$\bfcz'=\big(\pi^1_{\bfz_{0}}(\otimes\bfcs'),...,\pi^d_{\bfz_{0}}(\otimes\bfcs')\big)$ and $w'=w_0$. Notice that $\bfcz'$
is a vector strong subtree of $\bfcz$ of height $N$. Moreover, the inclusion $\otimes\bfcs'\subseteq \Delta_0$ and the
definition of $\Delta_0$ yield that the level selection $D$ is $(w',\bfcz',\beta+\varrho^2/2)$-dense. So this case implies
part (I) of the lemma.

Now assume that part (ii) of Corollary \ref{comb1} holds true. As the reader might have already guess, we will show 
that part (II) of the lemma is satisfied. Specifically, by our assumption, there exist $w''\in A$ and  
$\mathcal{N}^{*}\subseteq \{0,...,h-1\}$ with $|\mathcal{N}^{*}|\meg \varrho^3 h(\bfcs)$ such that, setting 
$\alpha'=\alpha-\gamma_0-\gamma_1-\gamma_2$ and defining $\Delta^*\subseteq \otimes\bfcs$ by the rule
\begin{eqnarray} \label{delta*}
\bfs\in \Delta^{*} & \Leftrightarrow & f_{\bfz}(\bfs,w)\meg \alpha' \text{ and } g_\bfcf(\bfs,w)\mik \theta \lambda^{-3} 
\text{ for every } \bfcf\in \mathcal{F}, \\
& & \text{every } \bfz\in \otimes\bfcz(m+1) \text{ and every } w\in\immsuc_W(w'') \nonumber
\end{eqnarray}
we have $|\Delta^{*}\cap \mathcal \otimes\bfcs(n)|\meg \varrho^3 h(\bfcs)$ for every $n\in\mathcal{N}^{*}$. 
Arguing precisely as above and using Theorem \ref{uniform dhl}, we see that there exists a vector strong subtree
$\bfcs''$ of $\bfcs$ with $h(\bfcs'')=N$ and such that $\otimes\bfcs''\subseteq \Delta^*$. For every $i\in [d]$ let
\begin{equation} Z''_i=(Z_i\upharpoonright m) \cup \{\pi^i_{\bfz}(\otimes\bfcs''):\bfz\in\otimes\bfcz(m+1)\big\} \end{equation}
and set $\bfcz''=(Z''_1,...,Z''_d)$. It is easy to check that $\bfcz''$ is a vector strong subtree of $\bfcz$ with
$\bfcz''\upharpoonright m= \bfcz\upharpoonright m$ and $h(\bfcz'')=(m+1)+N$. Also we set
\begin{equation} B=\{\bfz\in\otimes\bfcz(m):w''\in D(\bfz)\}. \end{equation}
Since $w''\in A\subseteq C$, by  (\ref{ee920}), we have $|B|\meg (\alpha/2)|\!\otimes \bfcz(m)|$. We will show that
the node $w''$, the vector strong subtree $\bfcz''$ and the set $B$ satisfy (II1), (II2) and (II3).

By the definition of the set $B$ we have $w''\in \bigcap_{\bfz\in B}D(\bfz)$. On the other hand, $w''\in A$. 
Thus, by the choice of the set $A$ in (\ref{ee92199}), we see that $w''\notin \bigcap_{\bfz\in\otimes\bfcg(1)}D(\bfz)$ 
for every $\bfcg \in\strong_2(\bfcz,m)$ with $\bfcg(0)\in \Gamma$. Therefore, part (II1) is satisfied. Next recall that
$\otimes\bfcs''\subseteq\Delta^*$. Hence, $f_{\bfz}(\bfs,w)\meg \alpha'$  for every $\bfz\in\otimes\bfcz(m+1)$, every
$\bfs\in\otimes\bfcs''$ and every $w\in\immsuc_W(w'')$. This is equivalent to say that the level selection $D$ is 
$(w'',\suc_{\bfcz''}(\bfz),\alpha-\gamma_0-\gamma_1-\gamma_2)$-strongly dense for every $\bfz\in\otimes\bfcz''(m+1)$. 
So, part (II2) is also satisfied. To verify part (II3), fix $\bfcg\in\strong_2(\bfcz'',n)$ for some 
$n\in\{m+1,..., h(\bfcz'')-1\}$ with $\bfcg(0)\in\Gamma$. Observe that there exist a unique $\bfcf\in \strong_2(\bfcz, m+1)$
with $\bfcf(0)=\bfcg(0)$ and a (not necessarily unique) $\bfs\in \otimes\bfcs''$ such that $\bfcg=\bfcg^\bfcf_\bfs$. Hence,
for every $w\in\immsuc_W(w'')$ we have
\begin{equation} \label{end-of-corollary}
\dens\Big( \bigcap_{\bfz\in\otimes\mathbf{G}(1)} D(\bfz) \ | \ w\Big)= 
\dens\Big( \bigcap_{\bfz\in\otimes\mathbf{G}^{\bfcf}_{\bfs}(1)} D(\bfz) \ | \ w\Big)=g_\bfcf(\bfs, w).
\end{equation}
Moreover, since $\otimes\bfcs''\subseteq\Delta^*$ we have $g_\bfcf(\bfs,w)<\theta\lambda^{-3}$ for every $w\in\immsuc_W(w'')$.
Combining the previous remarks we conclude that the pair $(\otimes\bfcg(1),w'')$ is strongly $(\theta\lambda^{-3})$-negligible
with respect to $D$, and so, part (II3) is satisfied. The proof of Lemma  \ref{c67} is thus completed.
\end{proof}
The following lemma is the second main result of this section.
\begin{lem} \label{c72}
Let $d\in\nn$ with $d\meg 1$ and $b_1,...,b_d, b_{d+1}\in\nn$ with $b_i\meg 2$ for all $i\in [d+1]$. Also let $m\in\nn$
and $0<\theta\mik 1$. Assume that we are given
\begin{enumerate}
\item[(a)] a homogeneous tree $W$ with $b_W=b_{d+1}$,
\item[(b)] a finite vector homogeneous tree $\bfcz=(Z_1,...,Z_d)$ with $b_{\bfcz}=(b_1,...,b_d)$,
\item[(c)] a level selection $D:\otimes\bfcz\to\mathcal{P}(W)$,
\item[(d)] a nonempty subset $B$ of $\otimes\bfcz(m)$ and a node $w\in W$ with $w\in\bigcap_{\bfz\in B} D(\bfz)$.
\end{enumerate}
Finally, let $N\in\nn$ with $N\meg 1$ and suppose that
\begin{equation} \label{ne76}
h(\bfcz) \meg (m+1)+ \mil\big(\underbrace{b_1,...,b_1}_{b_1-\mathrm{times}}, ...,
\underbrace{b_d,...,b_d}_{b_d-\mathrm{times}}|N, 1, b^{(\prod_{i=1}^d b_i)^m}\big).
\end{equation}
Then, either
\begin{enumerate}
\item[(i)] there exists $\bfcf\in\strong_2(\bfcz)$ with $\bfcf(0)\in B$   such that the pair $(\bfcf, w)$ is strongly
$\theta$-correlated with respect to $D$, or 
\item[(ii)] there exist a vector strong subtree $\bfcz'$ of $\bfcz$ with $\bfcz'\upharpoonright m=\bfcz\upharpoonright m$ 
and $h(\bfcz')=(m+1)+N$, a subset $\Gamma$ of $B$ with $|\Gamma|\meg (1/b_{d+1})|B|$ and $p_0\in\{0,...,b_{d+1}-1\}$
such that the pair $(\otimes\bfcg(1), w^{\con_W}\!p_0)$ is $\theta$-negligible with respect to $D$ for every 
$\bfcg\in\strong_2(\bfcz')$ with $\bfcg(0)\in\Gamma$.
\end{enumerate}
\end{lem}
\begin{proof}
Let
\begin{equation} \mathcal{G}=\{\bfcg\in \strong_2(\bfcz):\bfcg(0)\in B\}. \end{equation}
By condition (d), we see that $w\in D(\bfcg(0))$ for every $\bfcg\in\mathcal{G}$. Assume that part (i) of the lemma
is not satisfied. This has, in particular, the following consequence.
\medskip

\noindent \textbf{(H):} \textit{for every $\bfcg\in\mathcal{G}$ there exists $p\in\{0,...,b_{d+1}-1\}$, depending
possibly on the choice of $\mathbf{G}$, such that the pair $(\otimes\bfcg(1), w^{\con_W}\!p)$ is $\theta$-negligible
with respect to $D$.}
\medskip

\noindent We will use hypothesis \textbf{(H)} to show that part (ii) is satisfied. To this end we argue as follows.
We set $h=h(\bfcz)-(m+1)$. For every $i\in [d]$ and every $z\in Z_i(m)$ the finite homogeneous trees $b_i^{<h}$ and 
$\suc_{Z_i}(z)$ have the same branching number and the same height. Therefore, as we described in \S 2.5, we may consider
the canonical isomorphism $\ci_z:b_i^{<h}\to \suc_{Z_i}(z)$. Notice that the canonical isomorphism $\ci_z$ induces a map 
$\Phi_z:\strong^0_2(b_i^{<h})\to \strong^0_2 \big(\suc_{Z_i}(z)\big)$ defined by
\begin{equation} \Phi_z(F)=\{z\}\cup \{\ci_z(u):u\in F(1)\}. \end{equation}
These remarks can, of course, be extended to the higher-dimensional setting. Specifically, set $\bfcu=(b_1^{<h},...,b_d^{<h})$
and let $\bfz=(z_1,...,z_d)\in\otimes\bfcz(m)$ be arbitrary. Define $\Phi_{\bfz}:\strong^0_2(\bfcu)\to\strong^0_2\big(\suc_{\bfcz}(\bfz)\big)$ by
\begin{equation} \label{e-section6-1}
\Phi_{\bfz}\big((F_1,...,F_d)\big)= \big(\Phi_{z_1}(F_1),...,\Phi_{z_d}(F_d)\big).
\end{equation}
Notice that for every $\bfcf\in\strong^0_2(\bfcu)$ and every $\bfz\in B$ we have that $\Phi_{\bfz}(\bfcf)\in\mathcal{G}$. 
This observation and hypothesis \textbf{(H)} isolated above permit us to define a coloring $c:\strong^0_2(\bfcu)\to \{0,...,b_{d+1}-1\}^B$
by the rule
\begin{eqnarray} \label{e-section6-2}
c(\bfcf)=(p_{\bfz})_{\bfz\in B} & \Leftrightarrow & p_{\bfz}=\min\{ p: \text{\textbf{(H)} is satisfied for } \Phi_{\bfz}(\bfcf) \text{ and } p\} \\
& & \text{for every } \bfz\in B. \nonumber
\end{eqnarray}
Next observe that $b_{\bfcu}=(b_1,...,b_d)$ and $|B|\mik |\!\otimes\bfcz(m)|=\prod_{i=1}^d b_i^m$. Moreover,
\begin{equation} \label{e-section6-3}
h(\bfcu)=h=h(\bfcz)-(m+1) \stackrel{(\ref{ne76})}{\meg} 
\mathrm{Mil}\big(\underbrace{b_1,...,b_1}_{b_1-\mathrm{times}}, ..., \underbrace{b_d,...,b_d}_{b_d-\mathrm{times}}|N,1, b^{|B|}\big).
\end{equation}
Therefore, by Corollary \ref{c24}, it is possible to find $\bfcu'=(U_1',...,U_d')\in\strong^0_N(\bfcu)$ and 
$(p_{\bfz})_{\bfz\in B}\in \{0,...,b_{d+1}-1\}^B$ such that $c(\bfcf)=(p_{\bfz})_{\bfz\in B}$ for every $\bfcf\in\strong^0_2(\bfcu')$.
By the classical pigeonhole principle, there exist a subset $\Gamma$ of $B$ with $|\Gamma|\meg (1/b_{d+1})|B|$ and
$p_0\in\{0,...,b_{d+1}-1\}$ such that $p_{\bfz}=p_0$ for every $\bfz\in\Gamma$. For every $i\in [d]$ let
\begin{equation} \label{e-section6-3new}
Z_i'= (Z_i\upharpoonright m) \cup \bigcup_{z\in Z_i(m)} \ci_z(U'_i)
\end{equation}
and define $\bfcz'=(Z_1',...,Z_d')$. We will show that with these choices part (ii) of the lemma is satisfied. 
Indeed, notice first that $\bfcz'$ is a strong subtree of $\bfcz$ of height $(m+1)+N$ and with 
$\bfcz'\upharpoonright m=\bfcz\upharpoonright m$. Now let $\bfcg\in\strong_2(\bfcz')$ with $\bfcg(0)\in\Gamma$ be arbitrary. 
Observe that there exists $\bfcf\in\strong^0_2(\bfcu')$ such that $\bfcg=\Phi_{\bfcg(0)}(\bfcf)$. Since $\bfcg(0)\in\Gamma$
we have $p_{\bfcg(0)}=p_0$ and so the pair $(\otimes\bfcg(1), w^{\con_W}\!p_0)$ is $\theta$-negligible with respect to $D$.
The proof of Lemma \ref{c72} is thus completed.
\end{proof}


\section{Proof of Theorem \ref{it5}}

\numberwithin{equation}{section}

In this section we complete the proof of Theorem \ref{it5} following the analysis outlined in \S 5. 
It is organized as follows.  In \S 7.1 we define certain numerical parameters. In \S 7.2 we state the main step
towards the proof of Theorem \ref{it5}, Lemma \ref{l91} below. The proof of Lemma \ref{l91} occupies the bulk
of this section and is given in \S 7.3. The proof of Theorem \ref{it5} is then completed in \S 7.4. 
\textit{Finally let $d\in\nn$ with $d\meg 1$, $b_1,...,b_d,b_{d+1}\in\nn$ with $b_i\meg 2$ for all $i\in [d+1]$
and $0<\ee\mik 1$. These data will be fixed throughout this section}.

\subsection{Defining certain parameters}

We set
\begin{equation} \label{e91}
K=\udhl\big(b_1,...,b_d|2,\ee/(4b_{d+1})\big) \text{ and } r=\Big( \frac{\ee}{48b_{d+1}(\prod_{i=1}^d b_i)^{K}}\Big)^{2^{3K-1}}.
\end{equation}
Recall that $K$ is the number of iterations of the algorithm described in \S 5. On the other hand, the quantity
$r$ will be used to control the density increment. Also let
\begin{equation} \label{e93}
Q=\frac{\big( \prod_{i=1}^d b_i^{b_i} \big)^{K} -\big( \prod_{i=1}^d b_i \big)^{K}}{\prod_{i=1}^d b_i^{b_i} -\prod_{i=1}^d b_i}
\end{equation}
and define
\begin{equation} \label{eth}
\theta(b_1,...,b_d,b_{d+1}|\ee)=\Big(\frac{\ee}{24 b_{d+1}Q}\Big)^{3 K-2}.
\end{equation}
Next let $f_1, f_2:\nn\to\nn$ be defined by
\begin{equation} \label{e95}
f_1(n) =  \lceil 1/r^3\rceil \udhl\big( \underbrace{b_1,...,b_1}_{b_1^{K}-\mathrm{times}}, ..., 
\underbrace{b_d,...,b_d}_{b_d^{K}-\mathrm{times}}|n,r^3\big)
\end{equation}
and
\begin{equation} \label{e96}
f_2(n) = \mil\big(\underbrace{b_1,...,b_1}_{b_1-\mathrm{times}}, ..., 
\underbrace{b_d,...,b_d}_{b_d-\mathrm{times}}|n, 1,b^{(\prod_{i=1}^d b_i)^{K-1}}\big)
\end{equation}
for every integer $n\meg 1$, while $f_1(0)=f_2(0)=0$. Finally, define $f:\nn\to\nn$ by
\begin{equation} \label{e98}
f(n)= (f_1\circ f_2)(n)+1
\end{equation}
for every $n\in\nn$.

\subsection{The main dichotomy}

We have the following.
\begin{lem} \label{l91}
Let $\bfct$ be a finite vector homogeneous tree with $b_{\bfct}=(b_1,...,b_d)$, $W$ a homogeneous tree with
$b_W=b_{d+1}$ and $D:\otimes\bfct\to\mathcal{P}(W)$ a level selection of density $\ee$. Let $N\in\nn$ with $N\meg 1$
and assume that
\begin{equation} \label{e911}
h(\bfct)\meg f^{(K)}(N).
\end{equation}
Then, either
\begin{enumerate}
\item[(i)] there exist a vector strong subtree $\bfct'$ of $\bfct$ with $h(\bfct')=N$ and $w'\in W$ such that
the level selection $D$ is $(w',\bfct',\ee+r^2/2)$-dense, or
\item[(ii)] there exist $\bfcf\in\strong_2(\bfct)$ and $w\in W$ such that the pair $(\bfcf, w)$ is strongly
$\theta(b_1,...,b_d,b_{d+1}|\ee)$-correlated with respect to $D$.
\end{enumerate}
\end{lem}

\subsection{Proof of Lemma \ref{l91}}

Before we proceed to the details we need, first, to do some preparatory work. For notational convenience
we shall denote the parameter $\theta(b_1,...,b_d,b_{d+1}|\ee)$  simply by $\theta$. We set
\begin{equation} \label{equation lambda}
\lambda=\frac{\ee}{24b_{d+1}Q}
\end{equation}
where $Q$ is defined in (\ref{e93}). Also, for every $n\in [K]$ let
\begin{equation}\label{thseq}
\theta_n=\theta\lambda^{-3(n-1)}.
\end{equation}
Notice that $\theta_1=\theta$ and $\theta_{n+1}=\theta_n\lambda^{-3}$ for every $n\in[K-1]$. Moreover, 
recursively we define two finite sequences $(\delta_n)_{n=0}^{3K-1}$ and $(\ee_n)_{n=0}^{K}$ of reals by the rule
\begin{equation} \label{e99}
\left\{ \begin{array} {l} \delta_0=r, \\ \delta_{n+1} =(\delta_n+\delta_n^2)^{1/2} \end{array}  \right. \text{ and } \ \
\left\{ \begin{array} {l} \ee_0=\ee, \\ \ee_{n+1} =\ee_n-(\delta_{3n}+\delta_{3n+1}+\delta_{3n+2}). \end{array}  \right.
\end{equation}
We will need the following elementary properties satisfied by these sequences.
\begin{enumerate}
\item[($\mathcal{P}$1)] For every $n\in\{0,...,3K-1\}$ we have $\delta_n\mik 2 r^{2^{-n}}$.
\item[($\mathcal{P}$2)] For every $n\in\{0,...,3K-2\}$ we have $\sum_{i=0}^{n} \delta_i+r^2 =\delta_{n+1}^2$.
\item[($\mathcal{P}$3)] We have $\sum_{n=0}^{3K-1} \delta_n \mik \ee/2$.
\item[($\mathcal{P}$4)] For every $n\in\{0,...,K\}$ we have $\ee_{n}=\ee-\sum_{i=0}^{3n-1} \delta_i$.
\item[($\mathcal{P}$5)] For every $n\in\{0,...,K\}$ we have $\ee/2 \mik \ee_n\mik \ee$.
\item[($\mathcal{P}$6)] For every $n\in\{0,...,K-1\}$ we have
\begin{equation} \label{e910}
\delta_{3n}\mik \delta_{3K-3} \mik 2 \Big( \frac{\ee}{48b_{d+1}(\prod_{i=1}^d b_i)^{K}}\Big)^4 \mik
\Big( \frac{\ee_n}{12b_{d+1}(\prod_{i=1}^d b_i)^{n+1}}\Big)^4.
\end{equation}
\end{enumerate}
The verification of these properties is left to the reader. We notice, however, that properties ($\mathcal{P}$3)
and ($\mathcal{P}$6) follow by the choice of $r$ in \eqref{e91}.

After this preliminary discussion we are ready to proceed to the details. We will argue by contradiction. 
In particular, recursively and assuming that neither (i) nor (ii) are satisfied, we shall construct
\begin{enumerate}
\item[(a)] a finite sequence $(\bfcz_n)_{n=1}^{K}$ of vector strong subtrees of $\bfct$,
\item[(b)] two finite sequences $(w_n)_{n=1}^{K}$  and $(\tilde{w}_n)_{n=1}^{K}$ of nodes of $W$, and
\item[(c)] a finite sequence $(\Gamma_n)_{n=1}^{K}$ of subsets of $\otimes\bfct$
\end{enumerate}
such that, setting $\bfcz_0=\bfct$, the following conditions are satisfied for every $n\in [K]$.
\begin{enumerate}
\item[(C1)] We have $\bfcz_n\upharpoonright n-1=\bfcz_{n-1}\upharpoonright n-1$ and  $h(\bfcz_n)=n+f^{(K-n)}(N)$.
\item[(C2)] We have $\Gamma_n\subseteq \otimes \bfcz_{n-1}(n-1)$ and $|\Gamma_n|\meg (\ee_n/2b_{d+1})|\!\otimes\bfcz_{n-1}(n-1)|$.
\item[(C3)] The level selection $D$ is $(\tilde{w}_n,\suc_{\bfcz_n}(\bfz),\ee_n)$-dense for every $\bfz\in\otimes\bfcz_n(n)$.
\item[(C4)] If $\bfcg\in\strong_2(\bfcz_n,k)$ for some $k\in\{n,...,h(\bfcz_n)-1\}$ with $\bfcg(0)\in\Gamma_1\cup ...\cup\Gamma_n$,
then the pair $(\otimes\bfcg(1),\tilde{w}_n)$ is $\theta_n$-negligible with respect to $D$.
\item[(C5)] We have $w_n\in \bigcap_{\bfz\in\Gamma_n}D(\bfz)$.
\item[(C6)] If $n\meg 2$ and $\bfcg\in\strong_2(\bfcz_{n-1},n-1)$ with $\bfcg(0)\in \Gamma_1\cup ...\cup \Gamma_{n-1}$, 
then $ w_n\notin \bigcap_{\bfz\in\otimes\bfcg(1)}D(\bfz)$.
\end{enumerate}
We will present the general step of the recursive construction. The initial step, that is, the choice of 
$\bfcz_1, w_1,\tilde{w}_1$ and $\Gamma_1$, proceeds similarly taking into account that $D$ is $(W(0),\bfct,\ee_0)$-dense.

So fix some $m\in[K-1]$ and assume that the sequences $(\bfcz_n)_{n=1}^m$, $(w_n)_{n=1}^{m}$, $(\tilde{w}_n)_{n=1}^{m}$ 
and $(\Gamma_n)_{n=1}^{m}$ have been selected so that conditions (C1)-(C6) are satisfied. Notice, first, that $\bfcz_m$
is a vector strong subtree of $\bfct$. Therefore, there exists a unique $\ell_m\in L(D)$ such that $D(\bfz)\subseteq W(\ell_m)$
for every $\bfz\in\otimes\bfcz(m)$. By condition (C3), we see that $\ell_W(\tilde{w}_m)\mik\ell_m$. We set
\begin{equation}
\Gamma^{(m)}=\bigcup_{k=1}^{m}\Gamma_k.
\end{equation}
Invoking conditions (C1) and (C2), we see that $\Gamma^{(m)}\subseteq \otimes\bfcz_m\upharpoonright m-1$. Also let
\begin{equation} \label{neqM}
N_1=f_2\big(f^{(K-m-1)}(N)\big) \text{ and } N_2=f^{(K-m-1)}(N).
\end{equation}
Notice that
\begin{eqnarray} \label{e924}
\ \ \ \ \ \ \ \ h(\bfcz_{m}) & \stackrel{(\mathrm{C1})}{=} & m + f^{(K-m)}(N) \\
& \stackrel{(\ref{e98})}{=} & (m+1) + (f_1\circ f_2) \big( f^{(K-m-1)}(N)\big) \nonumber \\
& \stackrel{(\ref{e95}),(\ref{neqM})}{\meg} & (m+1) + \frac{1}{r^3}\udhl\big(\underbrace{b_1,...,b_1}_{b_1^{m+1}-\mathrm{times}}, ..., 
\underbrace{b_d,...,b_d}_{b_d^{m+1}-\mathrm{times}}|N_1,r^3\big). \nonumber
\end{eqnarray}
Moreover,
\begin{eqnarray} \label{e925}
\gamma_0(\ee_{m},\ee,r) & \stackrel{(\ref{e61})}{=} & (\ee+r^2-\ee_{m})^{1/2} \stackrel{(\mathcal{P}4)}{=} \Big( \sum_{i=0}^{3m-1} \delta_i +r^2 \Big)^{1/2} \\
& \stackrel{(\mathcal{P}2)}{=} & \delta_{3m} \stackrel{(\ref{e910})}{\mik} \Big( \frac{\ee_{m}}{12b_{d+1}(\prod_{i=1}^d b_i)^{m+1}}\Big)^4 \nonumber.
\end{eqnarray}
Finally, by the choices of $Q$ and $\theta$ in (\ref{e93}) and (\ref{eth}) respectively and property ($\mathcal{P}$5), 
we see that
\begin{equation} \label{lth}
\lambda\mik\frac{\ee_m}{12b_{d+1}q(b_1,...,b_d,m)}
\end{equation}
and
\begin{equation} \label{lth-second}
\theta_m\mik\theta_{K}=\frac{\ee}{24b_{d+1}Q}<\frac{2\ee_m}{5b_{d+1}q(b_1,...,b_d,m)}
\end{equation}
where $q(b_1,...,b_d,m)$ is defined in (\ref{e214}).
\begin{claim} \label{main dichotomy-claim1}
There exist a vector strong subtree $\bfcz'$ of $\bfcz_m$ with $\bfcz'\upharpoonright m=\bfcz\upharpoonright m$
and $h(\bfcz')=(m+1)+N_1$, $w\in W(\ell_m)\cap\suc_W(\tilde{w}_m)$ and $B\subseteq\otimes \bfcz_m(m)$ with
$|B|\meg (\ee_m/2)|\!\otimes \bfcz_m(m)|$ satisfying the following properties.
\begin{enumerate}
\item[(1)] We have $w\in \bigcap_{\bfz\in B}D(\bfz)$. On the other hand, $w\notin \bigcap_{\bfz\in\otimes\bfcg(1)}D(\bfz)$
for every $\bfcg \in\strong_2(\bfcz_m,m)$ with 
$\bfcg(0)\in \Gamma^{(m)}$.
\item[(2)] $D$ is $(w,\suc_{\bfcz'}(\bfz),\ee_{m+1})$-strongly dense for every $\bfz\in\otimes\bfcz'(m+1)$.
\item[(3)] If $\bfcg\in\strong_2(\bfcz',k)$ for some $k\in\{m+1,...,h(\bfcz')-1\}$ with $\bfcg(0)\in \Gamma^{(m)}$,
the pair $(\otimes\bfcg(1),w)$ is strongly 
$\theta_{m+1}$-negligible with respect to $D$.
\end{enumerate}
\end{claim}
\begin{proof}[Proof of Claim \ref{main dichotomy-claim1}]
We will rely on Lemma \ref{c67}. Specifically, let ``$\alpha=\ee_m$", ``$\beta=\ee$", ``$\varrho=r$", $m$ be the
fixed integer, ``$\theta=\theta_m$", $\lambda$ be as in (\ref{equation lambda}), $W$ be our given homogeneous tree,
``$\bfcz=\bfcz_m$", ``$D=D\upharpoonright \otimes\bfcz_m$", ``$\ell=\ell_m$", ``$\tilde{w}=\tilde{w}_m$", 
``$\Gamma=\Gamma^{(m)}$"  and ``$N= N_1$". It is easy to check, using  what we have mentioned before the statement
of the claim and our inductive hypotheses, that Lemma \ref{c67} can be applied for these data. Noticing that $N_1\meg N$,
we see that if the first alternative of Lemma \ref{c67} holds true, then part (i) of Lemma \ref{l91} is satisfied. 
This, of course, contradicts our assumptions. Therefore, part (ii) of Lemma \ref{c67} is satisfied. The proof of the
claim will be completed once we show that
\begin{enumerate}
\item[(a)] $\ee_{m+1}=\ee_m-\gamma_0(\ee_m,\ee,r)-\gamma_1(\ee_m,\ee,r)-\gamma_2(\ee_m,\ee,r)$ and
\item[(b)] $\theta_{m+1}=\theta_m\lambda^{-3}$.
\end{enumerate}
Indeed, the equality in (b) above follows immediately by (\ref{thseq}). Moreover,
\begin{eqnarray} \label{e926}
\gamma_1(\ee_{m},\ee,r) & \stackrel{(\ref{e62})}{=} & \big(\gamma_0(\ee_m,\ee,r)+\gamma_0(\ee_m,\ee,r)^2\big)^{1/2} \\
& \stackrel{(\ref{e925})}{=} & (\delta_{3m}+\delta_{3m}^2)^{1/2}\stackrel{ (\ref{e99})}{=}\delta_{3m+1} \nonumber
\end{eqnarray}
and
\begin{eqnarray} \label{e926new}
\gamma_2(\ee_{m},\ee,r) & \stackrel{(\ref{e63})}{=} & \big(\gamma_1(\ee_m,\ee,r)+\gamma_1(\ee_m,\ee,r)^2\big)^{1/2} \\
& \stackrel{(\ref{e926})}{=} & (\delta_{3m+1}+\delta_{3m+1}^2)^{1/2}\stackrel{(\ref{e99})}{=}\delta_{3m+2}. \nonumber
\end{eqnarray}
Therefore, by the choice of $\ee_{m+1}$ in (\ref{e99}) and equalities (\ref{e926}) and (\ref{e926new}), we conclude
that the equality in (a) above is also satisfied. The proof of Claim \ref{main dichotomy-claim1} is thus completed.
\end{proof}
\begin{claim} \label{main dichotomy-claim2}
Let $\bfcz'$, $w$ and $B$ be as in Claim \ref{main dichotomy-claim1}. Then there exist a vector strong subtree
$\bfcz''$ of $\bfcz'$ with $\bfcz''\upharpoonright m=\bfcz'\upharpoonright m$ and $h(\bfcz'')=(m+1)+N_2$, a subset
$\Gamma$ of $B$ with $|\Gamma|\meg (1/b_{d+1})|B|$ and $p_0\in\{0,...,b_{d+1}-1\}$ such that the pair
$(\otimes\bfcg(1), w^{\con_W}\!p_0)$ is $\theta_{m+1}$-negligible with respect to $D$ for every $\bfcg\in\strong_2(\bfcz'')$
with $\bfcg(0)\in\Gamma$.
\end{claim}
\begin{proof} [Proof of Claim \ref{main dichotomy-claim2}]
Now we will rely on Lemma \ref{c72}. To this end notice that
\begin{eqnarray} \label{e930}
h(\bfcz') & =& (m+1)+N_1 \stackrel{(\ref{neqM})}{=} (m+1) + f_2( N_2)  \\
& \stackrel{(\ref{e96})}{\meg} & (m + 1)+ \mil\big(\underbrace{b_1,...,b_1}_{b_1-\mathrm{times}}, ..., 
\underbrace{b_d,...,b_d}_{b_d-\mathrm{times}}|N_2, 1,b^{(\prod_{i=1}^d b_i)^{m}}\big). \nonumber
\end{eqnarray}
Moreover, by Claim \ref{main dichotomy-claim1}, we have $B\subseteq \otimes \bfcz_m(m)$ and 
$w\in \bigcap_{\bfz\in B}D(\bfz)$. Therefore, we may apply Lemma \ref{c72} for the fixed integer $m$,
``$\theta=\theta_{m+1}$", our given homogeneous tree $W$, ``$\bfcz=\bfcz'$", ``$D=D\upharpoonright \otimes\bfcz'$", 
the set $B$, the node $w$ and ``$N=N_2$".  The first alternative of Lemma \ref{c72} yields that there exists 
$\bfcf\in\strong_2(\bfcz')$ such that the pair $(\bfcf, w)$ is strongly $\theta_{m+1}$-correlated with respect
to $D$. Noticing that $\theta_{m+1}\meg \theta(b_1,...,b_d,b_{d+1}|\ee)$ and invoking our hypothesis that part
(ii) of Lemma \ref{l91} is not satisfied, we see that the second alternative of Lemma \ref{c72} holds true. 
This readily gives the conclusion of the claim.
\end{proof}
We are in the position to define all necessary data for the general step of the recursive construction.
Specifically, let $w$ be as in Claim \ref{main dichotomy-claim1} and $\bfcz''$, $p_0$ and $\Gamma$ be as in Claim 
\ref{main dichotomy-claim2}. We set ``$\bfcz_{m+1}=\bfcz''$", ``$w_{m+1}=w$", ``$\tilde{w}_{m+1}=w^{\con_W}\!p_0$" 
and  ``$\Gamma_{m+1}=\Gamma$". It is easily seen that with these choices conditions (C1)-(C6) are satisfied. 
The recursive selection is thus completed.

Now we are ready to derive a contradiction. Notice first that, by condition (C1), we have $h(\bfcz_K)=K+N\meg K$. 
We set $\bfcs=\bfcz_{K}\upharpoonright K-1$ and $\mathcal{D}=\Gamma_1\cup...\cup \Gamma_K$. Invoking condition (C1), 
we see that $\bfcs\upharpoonright n=\bfcz_{n}\upharpoonright n$ for every $n\in\{0,...,K-1\}$. Therefore, by condition
(C2), for every $n\in \{0,...,K-1\}$ we have
\begin{equation} \label{endnew1}
|\mathcal{D}\cap\otimes\bfcs(n)|=|\Gamma_{n+1}|\meg (\ee_{n+1}/2b_{d+1})|\!\otimes\bfcs(n)|
\stackrel{(\mathcal{P}5)}{\meg}(\ee/4b_{d+1})|\!\otimes\bfcs(n)|.
\end{equation}
Since $b_{\bfcs}=(b_1,...,b_d)$ and $h(\bfcs)=K$, by the choice of $K$ in (\ref{e91}) and Theorem \ref{uniform dhl},
there exist $\bfcg\in\strong_2(\bfcs)$ and $1\mik n_1< n_2\mik K$ such that 
\begin{equation} \bfcg(0)\in\Gamma_{n_1} \text{ and } \otimes\bfcg(1)\subseteq \Gamma_{n_2}. \end{equation}
The inclusion $\otimes\bfcg(1)\subseteq \Gamma_{n_2}$ implies, in particular, that
\begin{equation} \bfcg\in\strong_2(\bfcs, n_2-1)=\strong_2(\bfcz_{n_2-1}, n_2-1). \end{equation}
Thus, by condition (C6), we have
\begin{equation} \label{ef917} 
w_{n_2}\notin \bigcap_{\bfz\in\otimes\bfcg(1)} D(\bfz).
\end{equation}
On the other hand, however, by condition (C5), we get that
\begin{equation} \label{e944}
w_{n_2} \in  \bigcap_{\bfz\in\Gamma_{n_2}} D(\bfz) \subseteq \bigcap_{\bfz\in\otimes\bfcg(1)} D(\bfz).
\end{equation}
This is clearly a contradiction. The proof of Lemma \ref{l91} is thus completed.

\subsection{Proof of Theorem \ref{it5}}

Recall that the constant $\theta(b_1,...,b_d,b_{d+1}|\ee)$ has been defined in \eqref{eth}.  Let $K$, $r$ and $f$
be as in \eqref{e91} and \eqref{e98} respectively. We set $K'=K\lceil 2/r^2\rceil$ and we define
\begin{equation} \label{nearly-end}
\mathrm{StrCor}(b_1,...,b_d,b_{d+1}|\ee)=f^{(K')}(2).
\end{equation}
With these choices, Theorem \ref{it5} follows by Lemma \ref{l91} via a standard iteration.


\end{document}